\documentclass[11pt]{amsart}
\usepackage{inputenc}
\usepackage{enumerate}
\usepackage{amssymb}
\usepackage{amsmath}
\usepackage{mathrsfs}
\usepackage{amsthm}
\usepackage{tikz-cd}
\usepackage{mathpazo}

\usepackage{extarrows}
\usepackage{color}
\usepackage{setspace}
\usepackage[colorlinks,linkcolor = blue,anchorcolor = red,urlcolor = blue,citecolor = blue]{hyperref}
\usepackage{graphicx, subfig}
\usepackage[square, comma, sort&compress, numbers]{natbib}

\theoremstyle{definition}
\newtheorem{Def}{Definition}[section]

\newtheorem{Rem}[Def]{Remark}
\theoremstyle{plain}
\newtheorem{Thm}[Def]{Theorem}
\newtheorem{Lem}[Def]{Lemma}
\newtheorem{Pro}[Def]{Proposition}
\newtheorem{Cor}[Def]{Corollary}

\newtheorem*{CNCon}{The coarse Novikov Conjecture}

\usepackage[left=2.6cm, right=2.6cm, top=3cm, bottom=3cm]{geometry}
\usepackage[all]{xy}
\usepackage{bm}

\def\IN{\mathbb N}\def\IR{\mathbb R}\def\IE{\mathbb E}\def\IC{\mathbb C}\def\IA{\mathbb A}

\def\A{\mathcal A}\def\C{\mathcal C}\def\B{\mathcal B}\def\K{\mathcal K}\def\Q{\mathcal Q}\def\H{\mathcal H}\def\S{\mathcal S}

\def\supp{\textup{supp}}

\def\prop{\textup{Prop}}

\def\Cl{\textup{Cliff}_{\IC}}
\def\wox{\widehat{\otimes}}
\def\ox{\otimes}

\def\wt{\widetilde}
\def\ox{\otimes}

\def\Ga{\Gamma}

\def\ICu{\mathbb C_{u,\infty}}
\def\CauL{C^*_{u,L,\infty}}
\def\Cau{C^*_{u,\infty}}
\def\ICuL{\mathbb C_{u,L,\infty}}
\def\PdG{(P_d(G_n))_{n\in\IN}}
\def\Yg{(Y_{d,\gamma}:\gamma\in\Gamma)}
\def\PdGA{(P_d(G_n),\mathcal A(V_n))_{n\in\IN}}
\def\Ai{A_{\infty}}
\def\AiL{A_{L,\infty}}

\author[J.~Deng]{Jintao Deng}
\address[J.~Deng]{Department of Mathematics, SUNY at Buffalo, NY 14260, USA.}
\email{jintaode@buffalo.edu}

\author[L.~Guo]{Liang Guo}
\address[L.~Guo]{Research Center for Operator Algebras, School of Mathematical Sciences, East China Normal University, Shanghai, 200241, P. R. China.}
\email{52205500015@stu.ecnu.edu.cn}

\author[Q.~Wang]{Qin Wang}
\address[Q.~Wang]{Research Center for Operator Algebras, School of Mathematical Sciences
\\ East China Normal University, Shanghai 200241, P. R. China. }
\email{qwang@math.ecnu.edu.cn}

\author[G.~Yu]{Guoliang Yu}
\address[G.~Yu]{Department of Mathematics, Texas A\&M University, College Station, TX, 77843, USA.}
\email{guoliangyu@tamu.edu}

\title{Higher index theory for spaces with an \emph{FCE-by-FCE} structure}
\thanks{The third author is partially supported by NSFC (No. 11831006, 12171156), the Science and Technology Commission of Shanghai Municipality (No. 22DZ2229014), and the fourth author is supported by NSF 2000082, 2247313 and the Simons Fellows Program.}
\date{}

\begin{document}
\maketitle

\begin{abstract}
Let $(1\to N_n\to G_n\to Q_n\to 1)_{n\in\IN}$ be a sequence of extensions of finite groups. Assume that the coarse disjoint unions of $(N_n)_{n \in \mathbb{N}}$, $(G_n)_{n \in \mathbb{N}}$ and $(Q_n)_{n \in \mathbb{N}}$ have bounded geometry. The sequence $(G_n)_{n\in\IN}$ is said to have an \emph{FCE-by-FCE structure}, if the sequence $(N_n)_{n\in\IN}$ and the sequence $(Q_n)_{n\in\IN}$ admit \emph{a fibred coarse embedding} into Hilbert space. In this paper, we show that the coarse Novikov conjecture holds for spaces with an FCE-by-FCE structure.
\end{abstract}

\tableofcontents

\section{Introduction}

The \emph{coarse Novikov conjecture} provides an algorithm to determine the non-vanishing of the higher index of an elliptic operator on a complete non-compact Riemannian manifold. For a proper metric space $X$, this conjecture claims that the coarse assembly map
\begin{equation}\label{assembly map}\mu:KX_*(X)\to K_*(C^*(X))\end{equation}
from the coarse $K$-homology of $X$ to the $K$-theory of the \emph{Roe algebra} of $X$ is injective. It has fruitful applications in geometry and topology. One of the important applications is that it implies the Gromov-Lawson conjecture which claims that uniformly contractive manifolds do not carry an Riemannian metric with uniformly positive scalar curvature. The reader is referred to \cite{Yu2006} for a survey on this conjecture and some relevant results.

The coarse Novikov conjecture has been verified for a large class of metric spaces. It turns out that the notion of \emph{coarse embedding} introduced by Gromov \cite{Gromov1993} is a powerful tool to study this conjecture. Recall that a metric space $(X,d_X)$ admits a coarse embedding into $(Y,d_Y)$ if there exist a map $f:X\to Y$ and two non-decreasing functions $\rho_1,\rho_2:[0,\infty)\to [0,\infty)$ such that $\lim_{t\to\infty}\rho_i(t)=\infty$, $(i=1,2)$ and
$$\rho_1(d_X(x,y))\leq d_Y(f(x),f(y))\leq \rho_2(d_X(x,y)),$$
for all $x,y\in X$. The fourth author \cite{Yu2000} showed that the assembly map \eqref{assembly map} is an isomorphism when the space $X$ has bounded geometry and admits a coarse embedding into a Hilbert space.

A well-known obstruction to admit a coarse embedding into Hilbert space is coarsely containing an expander, then it admits no coarse embedding into Hilbert space (see for example, \cite[Section 13.2]{HIT2020}). Expanders are also an obstruction for the assembly map to be surjective. It is then natural to ask if coarse containing an expander is the only obstruction to admit a coarse embedding into Hilbert space. This question is answered by Arzhantseva and Tessera in \cite{AT2015}. They introduced the notion of relative expander as a counterexample which does not coarsely embed into any $L^p$-space for $p\in[1,\infty)$, and also does not coarsely contain any expander. These examples are constructed by using sequences of Cayley graphs of finite groups with uniformly bounded degrees. It is pointed out in \cite{DWY2023} that these sequences admit a \emph{CE-by-CE} extension structure. Recall that a sequence of group extension $(1\to N_n\to G_n\to Q_n\to 1)_{n\in\IN}$ is said to have \emph{CE-by-CE} extension structure if the sequence $(N_n)_{n\in\IN}$ with the induced metric from the word metrics of $(G)_{n\in\IN}$ and the sequence $(Q_n)_{n\in\IN}$ with the quotient metrics admit coarse embeddings into Hilbert spaces. Also in \cite{DWY2023}, the first, third, and fourth author show that the assembly map for $X$ is an isomorphism if $X$ is the coarse disjoint union of a sequence of groups that admits a \emph{CE-by-CE} structure. For the injectivity part, the condition on the quotient groups can be weakened to that $(Q_n)_{n\in\IN}$ admits a coarse embedding into a Hadamard manifold or an $\ell^p$ space for $1\leq p< \infty$, see \cite{GLWZ2023} for some relevant discussions.

In this paper, we are going to extend the results above to a much weaker assumption, called \emph{``FCE-by-FCE"} structure. The main theorem for this paper is as follows.

\begin{Thm}\label{main result}
Let $(1\to N_n\to G_n\to Q_n\to 1)_{n\in\IN}$ be a sequence of extensions of finite groups
with uniformly finite generating subsets. Equip $N=\bigsqcup_{n\in\IN}N_n$ with the induced metric from the word metrics of $(G)_{n\in\IN}$ and $Q=\bigsqcup_{n\in\IN}Q_n$ with the quotient metrics.
If both $N$ and $Q$ admit a \emph{\bf fibred coarse embedding} into Hilbert space, then the coarse Novikov conjecture holds for the coarse disjoint union $G=\bigsqcup_{n\in\IN}G_n$.
\end{Thm}

The notion of fibred coarse embedding was introduced by in \cite{CWY2013} to study a maximal version of the coarse Baum-Connes conjecture. see Definition \ref{FCE}. It can also be used to study the coarse Novikov conjecture, see \cite{FS2014, CWW2014, GLWZ2022}. Roughly speaking, a metric space admits a fibred coarse embedding into a Hilbert space if it can be coarse embedding into this Hilbert space piece-wise and the coarse embeddings for intersection parts of two pieces are isometric. We shall similarly say that a sequence of group extensions admits a \emph{FCE-by-FCE} structure if it satisfies the conditions in Theorem \ref{main result}. Since coarsely embeddable metric spaces automatically admit a fibred coarse embedding into Hilbert spaces, one can see that \emph{FCE-by-FCE} is much weaker than \emph{CE-by-CE}.


In this paper, we utilize a modified geometric Dirac-dual-Dirac method to give a more straightforward proof for the following:

\begin{Thm}
 Let $X=\bigsqcup_{n \in \mathbb{N}} X_n$ be a sparse metric space. If the space $X$ admits a fibred coarse embedding into Hilbert space, then the coarse Novikov conjecture holds for $X$. 
\end{Thm}

The basic idea to study the coarse Novikov conjecture for metric spaces which admits a fibred coarse embedding into specific space is to reduce the conjecture to a more manageable form, called the coarse Novikov conjecture at infinity. Building upon the insights from \cite{DWY2023}, we employ the concept of fibred coarse embedding for the quotient groups $(Q_n)_{n\in\mathbb{N}}$ to define both the twisted Roe algebra and the localization algebra at infinity for the sequence $(G_n)_{n\in\mathbb{N}}$. By by cutting and pasting argument, we can reduce this problem to the coarse Baum-Connes conjecture at infinity for the normal subgroups $(N_n)_{n\in\mathbb{N}}$. While the \emph{maximal} version of the coarse Baum-Connes conjecture at infinity for metric spaces that can be fibred coarsely into a Hilbert space has been examined in \cite{CWY2013}, our current focus is on a more refined version. Specifically, we demonstrate that the $K$-theory of the maximal Roe algebra at infinity is isomorphic to that of its reduced counterpart.

This paper is organized as follows. In Section \ref{sec: intro to CNC}, we briefly recall the coarse Novikov conjecture for a metric space with bounded geometry, In Section \ref{sec: K-amenable at infinifty}, we prove that for a sparse metric space which admits a fibred coarse embedding into a Hilbert, the maximal and the reduced Roe algebra at infinity have the same $K$-theory. In Section \ref{sec: twisted CBC at infty}, we define the twisted Roe algebra and the twisted localization algebra for $(G_n)_{n\in\IN}$ by using the fibred coarse embedding for the quotient groups $(Q_n)_{n\infty}$. We then prove the evaluation map between the twisted algebras induces an isomorphism on $K$-theory by using the fibred coarse embedding for the normal subgroups $(N_n)_{n\in\IN}$.
In Section \ref{sec: Proof of main result}, we define the Bott maps and show that the Bott map from the localization algebra to the twisted localization algebra induces an isomorphism on $K$-theory.

\section{Premilinaries on the coarse Novikov conjecture}\label{sec: intro to CNC}

In this section, we shall briefly recall some background of the coarse Novikov conjecture (cf. \cite{HR1993, HIT2020}), including the concepts of Roe algebras and $K$-homology of proper metric spaces.

Let $X,Y$ be \emph{proper} metric spaces, i.e., the closure of bounded subsets in $X, Y$ are compact. A map $f:X\to Y$ is said to be \emph{coarse} if for any $R>0$, there exists $S>0$ such that $d_Y(f(x),f(y))<S$ whenever $d_X(x,y)<R$ and $f$ is metric proper, i.e., the preimage of a bounded set is still bounded. Two maps $f,g:X\to Y$ are \emph{close} if there exists $r>0$ such that $d_Y(f(x),g(x))<r$ for any $x,y\in X$. A coarse map $f:X\to Y$ is a \emph{coarse equivalence} if there exists a coarse map $g:Y\to X$ such that $f\circ g$ and $g\circ f$ are close to $id_Y$ and $id_X$, respectively. We say two metric spaces are \emph{coarsely equivalent} if there exists a coarse equivalence between them. In coarse geometry, these two spaces are considered equivalent.

Consider a sequence of bounded metric spaces $(X_n,d_n)_{n\in\IN}$. A \emph{coarse disjoint union} of $(X_n)_{n\in\IN}$ is the set $\bigsqcup_{n\in\IN}X_n$ equipped with the metric $d$ such that \begin{itemize}
\item[(1)] $d$ restricted to $X_n$ coincides with $d_n$;
\item[(2)] $d(X_i,X_j)\to\infty$ as $i+j\to\infty$ and $i\ne j$. \end{itemize}
One can see that the identity map between any two metrics satisfying (1) and (2) is a coarse equivalence. As a result, the coarse disjoint union of a sequence of metric spaces is unique up to coarse equivalence.  In this paper, we shall always assume that $X$ is a coarse disjoint union of a sequence of finite metric spaces. For a sequence of finite space $(X_n,d_n)_{d\in\IN}$, we say $(X_n)_{n\in\IN}$ has \emph{uniformly bounded geometry} if for any $R>0$, there exists $N>0$ such that any ball of radius $r$ in $X_n$ contains at most $N$ points for any $n\in\IN$. Notice that the coarse disjoint union $X=\bigsqcup_{n\in\IN}X_n$ has \emph{bounded geometry} if the sequence $(X_n)_{n\in\IN}$ has uniform bounded geometry.

For a finitely generated group $G$, there is a canonical way to view it as a proper metric space by using the \emph{word length metric}, we recommend readers to refer to \cite{NY2012} for details. Notice that the coarse geometry of $G$ does not depend on the choice of the generating set. Moreover, let $(G_n)_{n\in\IN}$ be a sequence of finite groups. One can also check that if $(G_n)_{n\in\IN}$ has uniformly finite generating subsets, then $(G_n)_{n\in\IN}$ has uniformly bounded geometry under the word length metric associated with these generating sets.

A map $f:X\to Y$ is said to be a coarse embedding, if there exist two non-decreasing functions $\rho_1,\rho_2:[0,\infty)\to [0,\infty)$ with $\lim_{t\to\infty}\rho_i(t)=\infty$, $(i=1,2)$, such that
$$\rho_1(d_X(x,y))\leq d_Y(f(x),f(y))\leq \rho_2(d_X(x,y)),$$
for all $x,y\in X$. Coarse embeddability into specific spaces shows some properness of the metric of $X$, which is one of the most important tools for studying the coarse Novikov conjecture. One is referred to \cite{GLWZ2023} to see the geometric Dirac-dual-Dirac method and how coarse embedding works in this method.

In \cite{CWY2013}, X.~Chen, Q.~Wang, and G.~Yu introduced a notion of \emph{fibred coarse embedding} as a generalization of coarse embedding. It turns out that fibred coarse embedding is much more useful when studying coarse disjoint unions. We shall recall an equivalent definition of fibred coarse embedding as follows (see \cite[Definiton 5.4]{CWY2013}).

\begin{Def}[\cite{CWY2013}]\label{FCE}
Let $(X_n,d_n)_{n\in\IN}$ be a sequence of finite metric spaces with uniformly bounded geometry. The coarse disjoint union $X=\bigsqcup_{n\in\IN}X_n$ is said to admit a fibred coarse embedding into Hilbert space if there exists
\begin{itemize}
\item[(1)]a field of Hilbert spaces $(H_x)_{x\in X_n,n\in\IN}$ over $X$;
\item[(2)]a section $s:X_n\to\sqcup_{x\in X_n}H_x$ for all $n\in\IN$;
\item[(3)]two non-decreasing functions $\rho_+$ and $\rho_-$ from $[0,\infty)$ to $[0,\infty)$ with $\lim_{r\to\infty}\rho_{\pm}(r)=\infty$;
\item[(4)]a non-decreasing sequence of numbers $0\leq l_0\leq l_1\leq\cdots\leq l_n\leq\cdots$ with $\lim_{n\to\infty}l_n = \infty$.
\end{itemize}
such that for each $x\in X_n$ there exists a "trivialization"
$$t_x:(H_z)_{z\in B_{X_n}(x,l_n)}\to B_{X_n}(x,l_n)\times H$$
such that the restriction of $t_x$ to the fiber $H_z$ for any $z\in B_{X_n}(x,l_n)$ is an affine isometry $t_x(z):H_z\to H$, satisfying:
\begin{itemize}
\item[(1)]for any $z,z'\in B_{X_n}(x,l_n)$,
$$\rho_-(d(z,z'))\leq \|t_x(z)(s(z))-t_x(z')(s(z'))\|\leq \rho_+(d(z,z'));$$
\item[(2)]for any $x,y\subset B_{X_n}(x,l_n)\cap B_{X_n}(y,l_n)\ne \emptyset$, there exists an affine isometry $t_{xy}:H\to H$ such that $t_x(z)\circ t^{-1}_y(z)= t_{xy}$ for all $z\in B_{X_n}(x,l_n)\cap B_{X_n}(y,l_n)$.
\end{itemize}\end{Def}

Typical examples of spaces which admit a fibred coarse embeddings into Hilbert space can be constructed by Box spaces of residually finite groups and warped cones. We will not get into details, the reader is referred to \cite{CWW2013, SW2021} for more details.

For a metric space $X$, we denote by $C_0(X)$ the $C^*$-algebra of all continuous functions on $X$ vanishing at infinity. A \emph{geometric $X$-module} (or simply, $X$-module) is a separable Hilbert space $H_X$ equipped with a non-degenerated $*$-representation of $C_0(X)$. An $X$-module is \emph{ample} if the only element of $C_0(X)$ acting as a compact operator is zero.

\begin{Def}[\cite{Roe1993}]
Let $X$ be a proper metric space, and $H_X$ an ample $X$-module.
\begin{itemize}
\item[(1)] The \emph{support} of $T\in\mathcal B(H_X)$, denoted by $\supp(T)$, is defined to be the set of all points $(x,y)\in X\times X$ such that for all $f,g\in C_0(X)$ with $f(x)\ne0$ and $g(y)\ne0$, we have
$$fTg\ne0.$$
\item[(2)] The \emph{propagation} of $T\in\mathcal B(H_X)$ is defined by
$$\prop(T)=\sup\{d(x,y)\mid (x,y)\in\supp(T)\}.$$
\item[(3)] An operator $T\in\mathcal B(H_X)$ is said to have \emph{finite propagation} if $\prop(T)<\infty$.
\item[(4)] An operator $T\in\mathcal B(H_X)$ is \emph{locally compact} if the operators $fT$ and $Tf$ are compact operators on $H_X$ for all $f\in C_0(X)$.
\item[(5)] The \emph{algebraic Roe algebra}, denoted by $\IC[X,H_X]$ (or simply $\IC[X]$), is the $*$-subalgebra of $\mathcal B(H_X)$ of all locally compact operators with finite propagation
\item[(6)] The \emph{Roe algebra,} denoted by $C^*(X)$, is defined to be the norm closure of $\IC[X]$ under the norm in $\mathcal B(H_X)$.
\end{itemize}\end{Def}

Note that $\IC[X]$ is a $*$-algebra which, up to non-canonical isomorphisms, does not depend on the choice of ample $X$-modules (cf. \cite[Section 4.3]{HIT2020}). There is a standard approach to an ample $X$-module as follows. Take a separable infinite-dimensional Hilbert space $H_0$, and a countable dense subset $Z\subseteq X$. Then the Hilbert space $H_X:=\ell^2(Z)\otimes H_0$ is equipped with a natural pointwise multiplication action of $C_0(X)$ by restriction to $Z$. One can easily check that $H_X$ is an ample $X$-module.

\begin{Def}
Define $\IC_{\operatorname{f}}[X]$ to be the $*$-algebra of all bounded functions $T:Z\times Z\to\K:=\K(H_0)$, also viewed as $Z\times Z$ matrices, such that\begin{itemize}
\item[(1)]for any $S>0$, there exists $C>0$ such that for any bounded subset $B\subset X$ with the diameter $diam(B)<S$, we have that
$$\#\{(x,y)\in B\times B\cap Z\times Z\mid T(x, y)\ne 0\}\leq C$$
\item[(2)]there exists $L>0$ such that
$$\#\{y\in Z|T(x,y)\ne0\}<L,\qquad\#\{y\in Z\mid T(y,x)\ne 0\}<L$$
for all $x\in Z$;
\item[(3)]there exists $R>0$ such that $T(x,y)=0$ whenever $d(x,y)>R$ for $x,y\in Z$.
\end{itemize}
\end{Def}

Note that in general $\IC_{\operatorname{f}}[X]$ is a $*$-subalgebra of $\IC[X,H_X]$, where $H_X=\ell^2(Z)\otimes H_0$, and the norm completions of these two $*$-algebras are the Roe algebra $C^*(X)$. We will use $\IC_{\operatorname{f}}[X]$ to replace $\IC[X]$. It is obvious that $\IC_{\operatorname{f}}[X]$ is naturally equal to $\IC[X]$ when $X$ is a discrete metric space with bounded geometry. 

\begin{Def}
The \emph{localization algebra} of $X$, denoted by $C^*_L(X)$, is defined to be the closure of the set
$$\left\{f:\IR_+\to C^*(X)\mid f\text{ is bounded and uniformly continuous, }\prop(f(t))\to0\text{ as }t\to\infty\right\},$$
under the norm $\|f\|=\sup_{t\in\IR_+}\|f(t)\|_{C^*(X)}$ for each path $f$.
The \emph{$K$-homology} of $X$, denoted by $K_*(X)$, is defined to be
$$K_*(X)=K_*(C^*_L(X)).$$
\end{Def}

Notice that there is a natural evaluation map
$$ev:C^*_L(X)\to C^*(X),\qquad f\mapsto f(0).$$
The evaluation map induces a homomorphism on $K$-theory
$$ev_*:K_*(C^*_L(X))\to K_*(C^*(X)).$$
We would like to point out that the Roe algebra encodes the large-scale geometry of a metric space, which makes its $K$-theory hard to compute. On the other hand, K-homology is relatively more computable. Consequently, the map induced by the evaluation map on $K$-theory serves as a powerful tool for understanding the $K$-theory of the Roe algebra. 
However, the map $ev_*$ is not guaranteed to be an isomorphism in all cases. To define the assembly map, we shall also ``coarsen" the metric space in the left side by using the Rips complex.

\begin{Def}
Let $X$ be a discrete metric space with bounded geometry. For each $d\geq 0$, the \emph{Rips complex} $P_d(X)$ at scale $d$ is defined to be the simplicial polyhedron in which the set of vertices
is $X$, and a finite subset $\{x_0,x_1,\cdots,x_n\}\subseteq X$ spans a simplex if and only if $d(x_i,x_j )\leq d$ for all $0\leq i,j\leq n$.
\end{Def}

We shall define the spherical metric $d$ on $P_d(X)$. On each path connected component of $P_d(X)$, the spherical metric is the maximal metric whose restriction to each simplex $\{\sum_{i=0}^nt_ix_i\mid t_i\geq0,\sum_it_i=1\}$ is the metric obtained by identifying the simplex with $S^n_+$ via the map
$$\sum_{i=0}^nt_ix_i\to\left(\frac{t_0}{\sqrt{\sum_{i=0}^nt_i^2}},\cdots,\frac{t_n}{\sqrt{\sum_{i=0}^nt_i^2}}\right),$$
where $S^n_+=\{(x_0,\cdots,x_n)\in\IR^{n+1}\mid x_i\geq0,\sum_{i=0}^nx_i^2=1\}$ endowed with the standard Riemannian metric on the unit $n$-sphere. The distance between a pair of points in different connected components is defined to be infinity.

If $d<d'$, then $P_d(X)$ is included in $P_d(X)$ as a subcomplex via a simplicial map. By taking the inductive
limit, we obtain the assembly map
$$\mu:\lim_{d\to\infty}K_*(P_d(X))\to\lim_{d\to\infty}K_*(C^*(P_d(X)))\cong K_*(C^*(X)).$$

\begin{CNCon}
If $X$ is a discrete metric space with bounded geometry, then the assembly map $\mu$ is injective.
\end{CNCon}

\subsection{The coarse Novikov conjecture at infinity}\label{subsec: CNinfty}

In this subsection, we outline the strategy for proving Theorem \ref{main result}. This strategy involves a reduction of the coarse Novikov conjecture to its counterpart "at infinity", as originally introduced in \cite{CWY2013} and subsequently extended in \cite{GLWZ2022}.

Let $(1\to N_n\to G_n\to Q_n\to 1)_{n\in\IN}$ be a sequence of extensions of finite groups with uniformly finite generating subsets. Each $G_n$ is equipped with word length metric which naturally restricts to give a metric on the normal subgroup $N_n\subset G_n$, and each $Q_n$ is endowed with the quotient metric. Let $G=\bigsqcup_{n\in\IN}G_n$, $N=\bigsqcup_{n\in\IN}N_n$ and $Q=\bigsqcup_{n\in\IN}Q_n$ be the coarse disjoint unions. We say that a sequence $(1\to N_n\to G_n\to Q_n\to 1)_{n\in\IN}$ has an \emph{``FCE-by-FCE"} structure, if the coarse disjoint unions $N$ and $Q$ admit fibred coarse embeddings into Hilbert space.

For each $d>0$, we choose a countable dense subset $Z_d\subseteq P_d(G)$ with $Z_{d}\subseteq Z_{d'}$ for $d<d'$. Denote $Z_{d,n}=Z_d\cap P_d(G_n)$, for all $d\geq 0$, $n\in\IN$.

\begin{Def}\label{Roe algebra at infinity}
For each $d\geq0$, define $\ICu[\PdG]$ to be the set of all equivalence classes $T=[(T^{(0)},\cdots,T^{(n)},\cdots)]$ of sequences $(T^{(0)},\cdots,T^{(n)},\cdots)$ satisfying the following:
\begin{itemize}
\item[(1)]$(T^{(n)})_{n\in\IN}$ is a collection of uniformly bounded functions, where $T^{(n)}$ is a function from $Z_{d,n}\times Z_{d,n}$ to $\mathcal K$ for all $n\in\IN$;
\item[(2)]for any $S>0$, there exists $C>0$ such that for any bounded subset $B\subset P_d(G_n)$ with the diameter $diam(B)<S$, we have that
$$\#\{(x,y)\in B\times B\cap Z_{d,n}\times Z_{d,n}\mid T^{(n)}(x, y)\ne 0\}\leq C;$$
\item[(3)]there exists $L>0$ such that
$$\#\{y\in Z_{d,n}|T^{(n)}(x,y)\ne0\}<L,\qquad\#\{y\in Z_{d,n}\mid T^{(n)}(y,x)\ne 0\}<L$$
for all $x\in Z_{d,n}$, $n\in\IN$;
\item[(4)]there exists $R>0$ such that $T^{(n)}(x,y)=0$ whenever $d(x,y)>R$ for $x,y\in Z_{d,n}$, $n\in\IN$. The least such $R$ is called the propagation of the sequence $(T^{(0)},\cdots,T^{(n)},\cdots)$.
\end{itemize}
The equivalence relation $\sim$ on these sequences is defined by
$$(T^{(0)},\cdots,T^{(n)},\cdots)\sim(S^{(0)},\cdots,S^{(n)},\cdots)$$
if and only if
$$\lim_{n\to\infty}\sup_{x,y\in Z_{d,n}}\|T^{(n)}(x,y)-S^{(n)}(x,y)\|_{\K}=0.$$
\end{Def}

Viewing each $T^{(n)}$ as a $Z_{d,n}\times Z_{d,n}$-matrix, the product structure for $\ICu[\PdG]$ is
defined to be the usual matrix operations. Prior to introducing the norm on $\ICu[\PdG]$, we shall first discuss the ghost ideal of Roe algebras. Denoted by $\prod_{n\in\IN}^uC^*(P_d(G_n))$ the $C^*$-subalgebra of $\prod_{n\in\IN}C^*(P_d(G_n))$ generated by all sequences with uniformly finite propagation. An element $(T^{(n)})\in\prod_{n\in\IN}^uC^*(P_d(G_n))$ is called a \emph{ghost} operator, if for any fixed $R>0$, we have
$$\lim_{n\to\infty}\sup_{g,h\in G_n}\|\chi_{B(g,R)}T^{(n)}\chi_{B(h,R)}\|=0.$$
 All ghost operators in $\prod_{n\in\IN}^uC^*(P_d(G_n))$ forms a closed ideal, which is called the \emph{ghost ideal} and denoted by $I_G$. It is well-known that the ghost ideal contains all compact operators, and a ghost operator with finite propagation must be compact (see \cite{Roe2003}).

\begin{Def} 
Let $(G_n)_{n \in \mathbb{N}}$ be a sequence of bounded metric spaces with bounded geometry.
\begin{itemize}
    \item[(1)] The \emph{Roe algebra at infinity}, denoted by $C^*_{u,\infty}(\PdG)$, is defined to be the completion of $\ICu[\PdG]$ with respect to the norm
$$\|T\|=\inf\left\{\|(T^{(n)}+S^{(n)})_{n\in\IN}\|:\ (S^{(n)})_{n\in\IN}\text{ is a ghost operator}\right\}.$$
This norm is is referred to as the \emph{reduced norm}.

   \item[(2)] The \emph{maximal Roe algebra at infinity}, denoted by $C^*_{u,{\rm max},\infty}(\PdG)$, is defined to be the completion of $\ICu[\PdG]$ with respect to the norm
$$\|T\|_{\rm max}=\sup\{\|\phi(T)\|\mid \phi:\ICu[\PdG]\to\B(\H_{\phi})\text{ is a $*$-representation}\}.$$
This norm is called the \emph{maximal norm}.
\end{itemize}
\end{Def}

Based on the definitions, it is evident that the Roe algebra at infinity is, in fact, the quotient algebra of the Roe algebra over the ghost ideal.

By the universal property of the maximal norm, there is a canonical quotient $*$-homomorphism
$$\pi_{\infty}:C^*_{u,{\rm max},\infty}(\PdG)\to C^*_{u,\infty}(\PdG).$$

\begin{Def}\label{localization algebra}
Let $\ICuL[\PdG]$ be the set of all bounded, uniformly norm-continuous functions
$$f:\IR_+\to \ICu[\PdG]$$
such that $f(t)$ is of the form $f(t)=[f^{(0)}(t),\cdots,f^{(n)}(t),\cdots]$ and satisfies that there exists a bounded function $R(t):\IR_+\to \IR_+$ with $\lim_{t\to\infty}R(t)=0$ such that $(f^{(n)}(t))(g,h)=0$ whenever $d(g,h)>R(t)$ and $n\in\IN$.

The localization algebra at infinity, denoted by $\CauL(\PdG)$, is defined to be the norm completion of $\ICuL[\PdG]$, under the norm
$$\|f\|=\sup_{t\in\IR_+}\|f(t)\|.$$
\end{Def}

\begin{Rem}\label{Remark on localization algebra at infinity}
It was proved in \cite[Proposition 4.3]{CWY2013} (also in \cite[Section 5]{GLWZ2022}) that there is a canonical isomorphism
$$K_*(\CauL(\PdG))\cong\frac{\prod_{n=0}^{\infty}K_*(P_d(G_n))}{\bigoplus_{n=0}^{\infty}K_*(P_d(G_n))}.$$
Here, we utilize the reduced Roe algebras at infinity, rather than the maximal version, to define the localization algebras at infinity. Actually, one can also perform this with the maximal version, and it gives rise to the same group on the level of $K$-theory, due to a cutting and pasting method (cf. \cite{Yu1997}). 
\end{Rem}

There is a natural evaluation homomorphism $e: \CauL(\PdG)\to \Cau(\PdG)$ defined by $e(g)=g(0)$ for $g\in\CauL(\PdG)$. The assembly map at infinity is defined to be the inductive limit of the evaluation map:
$$\mu_{\infty}:\lim_{d\to\infty}K_*(\CauL(\PdG))\to\lim_{d\to\infty}K_*(\Cau(\PdG)).$$
There is a canonical quotient $C^*$-homomorphism
$$\Phi:C^*(P_d(G))\to\Cau(\PdG).$$
The interested reader can also find a detailed discussion of this quotient map in \cite[Section 3]{CWW2014} or \cite[Section 5.1]{GLWZ2022}.

For any $d\geq 0$, there exists $N_d\in N$ large enough such that $d(G_n,G_m)> d$ for all distinct $n,m\in N_d$. It follows that $P_d(G)=P_d(G_{N_d})\bigsqcup\left(\sqcup_{n\geq N_d}P_d(G_n)\right)$ and
$$K_*(P_d(G))=K_*(P_d(G_{N_d}))\bigoplus\prod_{n=N_d}^{\infty}K_*(P_d(G_n)).$$
By definitions, we have the following commutative diagram:
\begin{equation}\label{CNC to CNC at infty}\xymatrix{0\ar[d]&&0\ar[d]\\K_*(P_d(G_{N_d}))\oplus\bigoplus_{n=N_d}^{\infty}K_*(P_d(G_n))\ar[rr]\ar[d]&&K_*(\mathcal{K})\ar[d]\\
K_*(P_d(G))\ar[rr]^{\mu}\ar[d]&&K_*(C^*(P_d(G)))\ar[d]^{\Phi_*}\\
\frac{\prod_{n=0}^{\infty}K_*(P_d(G_n))}{\bigoplus_{n=0}^{\infty}K_*(P_d(G_n))}\ar[rr]^{\mu_{\infty}}\ar[d]&&K_*(C^*_{u,\infty}(P_d(G_n)))\\0&&}\end{equation}
Passing to the inductive limit as $d$ tends to $\infty$, the top horizontal map is an isomorphism. The inclusion $\K\to C^*(P_d(G))$
induces an injection on $K$-theory (cf. \cite[Proposition 2.10]{OY2009}). Remark \ref{Remark on localization algebra at infinity}, combined with a diagram chasing argument, implies that the injectivity of $\mu$ is a result of the injectivity of $\mu_{\infty}$. 

In the rest of this paper, we shall prove the following theorem which, as discussed above, implies Theorem \ref{main result}.

\begin{Thm}\label{CNC at infty}
Let $(1\to N_n\to G_n\to Q_n\to 1)_{n\in\IN}$ be a sequence of extensions of finite groups with uniformly finite generating subsets which admits a ``FCE-by-FCE" structure. Then the assembly map at infinity for $(G_n)_{n\in\IN}$ is injective. 
\end{Thm}

\section{Maximal and reduced Roe algebras at infinity for FCE spaces}\label{sec: K-amenable at infinifty}

Before proving Theorem \ref{CNC at infty}, we first study the $K$-theory of the Roe algebra at infinity for a sparse metric space which admits a fibred coarse embedding into Hilbert space. More precisely, we shall prove the following theorem.

\begin{Thm}\label{K-amenable at infinity for FCE spaces}
Let $(X_n)_{n\in\IN}$ be a sequence of finite metric spaces with uniform bounded geometry which admits a fibred coarse embedding into Hilbert space. Then the canonical quotient map
\[\pi:C^*_{u,{\rm max},\infty}((P_d(X_n))_{n\in\IN})\to\Cau((P_d(X_n))_{n\in\IN})\]
induces an isomorphism on $K$-theory, i.e., the homomorphism
\[\pi_*:K_*\left(C^*_{u,{\rm max},\infty}((P_d(X_n))_{n\in\IN})\right)\xrightarrow{\cong} K_*\left(\Cau((P_d(X_n))_{n\in\IN})\right)\]
is an isomorphism.
\end{Thm}

\subsection{(Maximal and reduced) twisted Roe algebras at infinity}\label{subsec: Twisted algebras at infinity}

Theorem \ref{K-amenable at infinity for FCE spaces} is an analogy of \cite[Theorem 1.4]{JR2013} at infinity. The strategy is to exploit a geometric Dirac-dual-Dirac argument. Throughout this section, we will consistently assume that $X=\bigsqcup_{n\in\IN}X_n$ admits a fibred coarse embedding into the Hilbert space $\H$ as in Definition \ref{FCE}.

To utilize the geometric Dirac-dual-Dirac argument, we recall the $C^*$-algebra associated to the infinite dimensional Euclidean space introduced by N.~Higson, G.~Kasparov and J.~Trout in \cite{HKT1998}. Let $V\subseteq \H$ be a finite-dimensional affine space. Denoted by $V^0$ the finite-dimensional linear subspace of $\H$ consisting of differences of elements in $V$.  Let $\Cl(V^0)$ be the complexified Clifford algebra on $V^0$ and $\mathcal{C}(V)$ the graded $C^*$-algebra of continuous functions from $V_a$ to $\Cl(V^0)$ vanishing at infinity. Let $\mathcal{S}=C_0(\IR)$, graded according to even and odd functions. Define the graded tensor product
$$\mathcal{A}(V)=\mathcal{S}\wox\mathcal{C}(V).$$
If $V_a \subseteq V_b\subseteq\H$, we have a decomposition $V_b=V_{ba}^0+ V_a$, where $V_{ba}^0$ is the orthogonal complement of $V_a^0$ in $V_b^0$. For each $v_b \in V_b$, we have a unique decomposition $v_b=v_{ba}+v_a$, where $v_{ba} \in V_{ba}^0$ and $v_a\in V_a$. Note that $\C(V_b)=\C(V_{ba}^0)\wox\C(V_a)$.

\begin{Def}[\cite{HKT1998}]
Let $V_a,V_b$ be finite-dimensional affine subspaces of $\H$.\begin{itemize}
\item[(1)]If $V_a \subseteq V_b$, we define $C_{ba}$ to be the Clifford algebra-valued function $V^0_{ba} \rightarrow \Cl(V_{ba}^0)$, $v \mapsto v\in  V_{ba}^0 \subset \Cl(V_{ba}^0)$. Let $X$ be the function of multiplication by $x$ on $\mathbb{R}$. It can be viewed as a degree one and unbounded multiplier of $\mathcal{S}$.  Define a homomorphism $\beta_{ba}: \mathcal{A}(V_a) \rightarrow \mathcal{A}(V_b)\cong\S\wox\C(V_{ba}^0)\wox\C(V_a)$ by
$$\beta_{ba}(g\wox h)=g(X\wox 1+1\wox C_{ba})\wox h$$
for all $g\in \mathcal{S}$ and $h\in\mathcal{A}(V_a)$, and $g(X\wox 1+1\wox C_{ba})$ is defined by functional calculus of $g$ on the unbounded, essentially self-adjoint operator $X\wox 1+1\wox C_{ba}$.
\item[(2)]If $V_a\subseteq V_b$, for any subset $O\subset \IR_+\times V_a$, define
$$\overline{O}^{\beta_{ba}}=\left\{(t+v_{ba}+v_a)\in\IR_+\times V_b: (\sqrt{t^2+\|v_{ba}\|^2},v_a)\in O\right\}.$$
\end{itemize}\end{Def}

For any finite dimensional affine subspace $V_a$ of $H$, the algebra $C_0(\IR_+\times V_a)$ is included in $\A(V_a)$ as its center, (see \cite[Remark 7.7]{GWY2018} or \cite[Example 3.4]{GLWZ2023} for another explanation). For any function $a\in\A(V_a)$, the support of $a$, denoted by $\supp(a)$, is the complement of all $(t,v)\in \mathbb{R}_+ \times V_a$ such that there exists $g\in C_0(\mathbb{R}_+ \times V_a)$ with $g(t,v) \neq 0 $ and $g\cdot a=0$. 

\begin{Rem}\label{composition}
To study fibred coarse embeddings, we still need to consider compatibility between the Bott map and an isometry. Let $t: V\to W$ be an isometric bijection between two finite-dimensional affine spaces. There exists a canonical $*$-homomorphism $t_*:\C(V)\to\C(W)$ defined by
$$(t_*(h))(w)=h(t^{-1}(w)),$$
which induces a $*$-homomorphism $t_*:\A(V)\to\A(W)$, here we abuse notation slightly and use the same symbol $t_*$ for the map $1\wox t_*$. In the case when $t$ is not bijective, this map is defined by the composition
$$\A(V)\xrightarrow{t_*}\A(t(V))\xrightarrow{\beta_{t(V),W}}\A(W).$$
We still denote it by $t_*$.

If $V_a \subset V_b \subset V_c$ are finite dimensional affine subspaces of $\H$, then we have $\beta_{cb} \circ \beta_{ba}=\beta_{ca}$ (cf. \cite{HKT1998}). If $W_a\subset W_b\subset V_c$ are also finite-dimensional affine subspaces of $\H$ such that there exists $t: W_b\to V_b$ which is an affine isometry from $W_b$ onto $V_b$ mapping $W_a$ onto $V_a$. Then the canonical $*$-homomorphism induced by $t: \A(W_a)\to\A(V_a)$ gives rise to the following commutative diagram
$$\xymatrix{
\A(W_a)\ar[r]^{\cong}_{t_*}\ar[d]_{\beta_{W_b,W_a}}&\A(V_a)\ar[r]^{\beta_{V_c,V_a}}\ar[d]^{\beta_{V_b,V_a}}&\A(V_c)\ar[d]^{=}\\
\A(W_b)\ar[r]^{\cong}_{t_*}&\A(V_b)\ar[r]^{\beta_{V_c,V_b}}&\A(V_c)}.$$
The commutative diagram is useful when defining the product structure for twisted Roe algebras at infinity.
\end{Rem}

Recall that the sequence $(X_n)_{n\in\IN}$ admits a fibred coarse embedding into $\H$. We define
$$V_n=\text{affine-span}\left\{t_x(z)(s(z))\mid x\in X_n,z\in B(x,l_n)\right\}.$$
For each $x\in X_n$, we define
$$W_k(x)=\text{affine-span}\left\{t_x(z)(s(z))\mid z\in Z_n\cap B(x,k)\right\}\subseteq V_n.$$
If $z\in B(x,l_n)\cap B(y,l_n)$, then there exists an isometry $t_{xy}:W_k(y)\to V_n$ (not surjective) such that
$$t_{xy}(t_y(z)(s(z)))=t_x(z)(s(z)).$$
As discussed earlier, for sufficiently large $n\in\IN$, there exists a canonical map $\beta_{W_k(x),V_n}:\A(W_k(x))\to\A(V_n)$ and $t_{xy}$ induces a canonical $*$-homomorphism
$$(t_{xy})_*:\A(W_k(x))\to\A(V_n).$$

Now we are ready to define the twisted algebras at infinity.

\begin{Def}\label{twisted Roe algebra for FCE}
The \emph{algebraic twisted Roe algebra at infinity}, denoted by $\ICu[(X_n,\A(V_n))_{n\in\IN}]$, is defined to be the set of all equivalent classes $T=[(T^{(0)},\cdots,T^{(n)},\cdots)]$ of sequences $(T^{(n)})_{n\in\IN}$ such that\begin{itemize}
\item[(1)] each $\{T^{(n)}$ is a bounded map from $ X_n\times X_n$ to $ A(V_n)\wox\K\}$ such that 
$$
\sup_{n \in \mathbb{N}}\sup_{x,y\in X_{n}}\|T^n_{x,y}\|_{A(V_n)\wox\K\}}< \infty;
$$
\item[(2)] there exists $k>0$ such that $T^{(n)}(x,y)$ is in the image of $\beta_{W_k(x), V_n}$ for any $x,y\in X_n$ and $n\in\IN$, and the preimage can be written as a finite linear combination of elementary tensors from $\S\wox\C(W_k(x))\wox\K$;
\item[(3)] there exists $R>0$ such that, for all $n\in\IN$, $T^{(n)}(x,y)=0$ for all $x,y\in X_n$ with $d(x,y)>R$;
\item[(4)]there exists $r>0$ such that
\[\begin{split}\supp(T^{(n)}(x,y))&\subseteq B_{\IR_+\times V_n}(t_x(x)(s(x)),r)\\
&:=\{(t,v)\in\IR_+\times V_n\mid t^2+\|v-t_x(x)(s(x))\|^2<r^2\}\end{split}\]
for any $x,y\in X_n$ and $n\in\IN$;
\item[(5)]there exists $c>0$ such that for any $w\in\IR_+\times W_k(x)$ with norm no more than one, if $T_1(x,y)\in\A(W_k(x))\wox\K$ such that $\beta_{W_k(x),V_n}(T_1(x,y))=T(x,y)$, then the derivative of the function $T_1(x,y)$ in the direction $w$, denoted by $\nabla_w(T_1(x,y))$, exists in $\A(W_k(x))\ox\K$ with $\|\nabla_w(T_1(x,y))\|\leq c$ for any $x,y\in X_n$ and $n\in\IN$.
\end{itemize}
The equivalence relation $\sim$ on sequences is defined by
$$\left(T^{(n)}\right)_{n\in\IN}\sim \left(S^{(n)}\right)_{n\in\IN}$$
if and only if
$$\lim_{n\to\infty}\sup_{x,y\in X_{n}}\|T^{(n)}(x,y)-S^{(n)}(x,y)\|_{\A(V_n)}=0.$$
\end{Def}

The multiplication on $\ICu[(X_n,\A(V_n))_{n\in\IN}]$ is defined as follows. For any two elements $T=[(T^{(0)},\cdots,T^{(n)},\cdots)]$ and $S=[(S^{(0)},\cdots,S^{(n)},\cdots)]$ in $\ICu[(X_n,\A(V_n))_{n\in\IN}]$, the product is defined to be
$$TS=[((TS)^{(0)},\cdots,(TS)^{(n)},\cdots)]$$
where there exists a sufficiently large $N\in\IN$ which depends on the propagation of $T$, such that $(TS)^{(n)}=0$ for $n < N$ and
$$(TS)^{(n)}(x,y)=\sum_{z\in X_{n}}\left(T^{(n)}(x,z)\right)\cdot\left((t_{xz})_*\left(S^{(n)}(z,y)\right)\right).$$
It is routine to check that this product is well-defined, and the reader is also referred to \cite[Remark 5.6]{CWY2013} for details.

The involution on $\ICu[(X_n,\A(V_n))_{n\in\IN}]$ is defined by
$$[(T^{(0)},\cdots,T^{(n)},\cdots)]^*=[((T^*)^{(0)},\cdots,(T^*)^{(n)},\cdots)]$$
where
$$(T^*)^{(n)}(x,y)=(t_{xy})_*\left(\left(T^{(n)}(y,x)\right)^*\right)$$
for all but finitely many $n$, and $0$ otherwise. Then $\ICu[(X_n,\A(V_n))_{n\in\IN}]$ is made into a $*$-algebra. 

The norm on $\ICu[(X_n,\A(V_n))_{n\in\IN}]$ closely resembles the one discussed in \cite[Remark 6.4]{GLWZ2022}, and the norm in our setting is more straightforward. For each $n\in\IN$, define
\begin{equation}\label{Hilbert module E_n}E_n=\left\{f:X_n\to C(X_n,\A(V_n)\wox\K)\ \Big|\ \supp(f(x))\subseteq B\left(x,\frac{l_n}2\right)\right\}.\end{equation}
Here, $C(X_n,\A(V_n)\wox\K)$ is the algebra of all bounded maps $b: X_n\to \A(V_n)\wox\K$, and this algebra can be further identified with $\bigoplus_{x\in X_n}\A(V_n)\wox\K$ under the corresponding $b\leftrightarrow \oplus_{x\in X_n}b(x)$. Then $E_n$ forms a Hilbert $\A(V_n)\wox\K$-module as follows. For any $a\in \A(V_n)\wox\K$ and $f,f_1,f_2\in E_n$, the right module action, and the inner product are given by
$$((fa)(x))(y)=(f(x))(y)\cdot a,$$
$$\langle f_1,f_2\rangle=\sum_{x,y\in X_n}(f_1(x))(y)^*\cdot(f_2(x))(y).$$
For any $[(T^{(n)})_{n\in\IN}]\in\ICu[(X_n,\A(V_n))_{n\in\IN}]$, we fix a representative  $(T^{(n)})_{n\in\IN}$. The operator $T^{(n)}$ on $E_n$ is defined by
$$\left(\left(T^{(n)}f\right)(x)\right)(y)=\sum_{z\in X_n}(t_{yx})_*\left(T^{(n)}(x,z)\right)\cdot (f(z))(y).$$
We remark here that for given $(T^{(n)})_{n\in\IN}$, there exists some $N>0$ dependent on its propagation such that the action of $T^{(n)}$ on $E_n$ is well-defined for any $n>N$, since $l_n$ tends to $\infty$ as in Definition \ref{FCE}.

\begin{Rem}\label{reduced norm is well-defined}
Here, we will provide some additional details regarding the right module action. It is evident that $T^{(n)}$ acts as a bounded module homomorphism, so we will focus on demonstrating how the $*$-structure of $T^{(n)}$ aligns with the inner product.

For given $T=[(T^{(n)})_{n\in\IN}]$ , there exists $R>0$ such that $T^{(n)}(x,z)=0$ whenever $d(x,z)>R$ for $x,z\in Z_{d,n}$, $n\in\IN$. Then for sufficiently large $n$, by the definition, we obtain that
\begin{equation}\begin{split}\label{original equation}
\left\langle f_1, T^{(n)}f_2\right\rangle&=\sum_{x,y\in X_n}(f_1(x))(y)^*\cdot\left(\sum_{z\in X_n}(t_{yx})_*\left(T^{(n)}(x,z)\right)\cdot (f_2(z))(y)\right)\\
&=\sum_{x,y,z\in X_n}(f_1(x))(y)^*\cdot\left((t_{yx})_*\left(T^{(n)}(x,z)\right)\right)\cdot (f_2(z))(y).
\end{split}\end{equation}
Similarly, we can also compute that
\begin{equation}\begin{split}\label{star equation}
\left\langle \left(T^*\right)^{(n)}f_1, f_2\right\rangle&=\sum_{x,y\in X_n}\left(\sum_{z\in X_n}(t_{yx})_*\left(\left(T^*\right)^{(n)}(x,z)\right)\cdot (f_1(z))(y)\right)^*\cdot(f_2(x))(y)\\
&=\sum_{x,y,z\in X_n}(f_1(z))(y)^*\cdot\left((t_{yx})_*\left(\left(T^*\right)^{(n)}(x,z)\right)\right)^*\cdot (f_2(x))(y).
\end{split}\end{equation}
Combining the $*$-structure of $\ICu[(X_n,\A(V_n))_{n\in\IN}]$ with the fact that $t_{yz}=t_{yx}\circ t_{xz}$ whenever the compositions are defined, we have that
\begin{equation}\begin{split}\label{compute star}
\left((t_{yx})_*\left(\left(T^*\right)^{(n)}(x,z)\right)\right)^*&=\left((t_{yx})_*\circ(t_{xz})_*\left(\left(T^{(n)}(z,x)\right)^*\right)\right)^*\\&=\left((t_{yz})_*\left(T^{(n)}(z,x)\right)\right)
\end{split}\end{equation}
By substituting \eqref{compute star} into \eqref{star equation} and interchanging the roles of $x$ and $z$, one can easily obtain that \eqref{star equation} is equal to \eqref{original equation}. Therefore, we have established the compatibility of the representation with the $*$-structure. Similarly, we can also verify that the representation is compatible with the multiplication.
\end{Rem}

Let $B^*((X_n,\A(V_n))_{n\in\IN})$ be the completion of the $*$-algebra $\ICu[(X_n,\A(V_n))_{n\in\IN}]$ under the norm 
$$\|(T^{(n)})_{n\in\IN}\|=\limsup_{n\to\infty}\|T^{(n)}\|_{E_n}.$$
An element $[(T^{(n)})_{n\in\IN}]\in B^*((X_n,\A(V_n))_{n\in\IN})$ is called a ghost if for any $R>0$, one has that
$$\lim_{n\to\infty}\sup_{x,y\in X_n}\|\chi_{B(x,R)}T^{(n)}\chi_{B(y,R)}\|=0.$$
We denote by $I_G(\A)$ the ideal of all ghost elements in $B^*((X_n,\A(V_n))_{n\in\IN})$.

\begin{Def}\label{norms on twisted Roe algebra at infinity for FCE}
The \emph{reduced twisted Roe algebra at infinity}, denoted by $\Cau((X_n,\A(V_n))_{n\in\IN})$, is defined to be the completion of $\ICu[(X_n,\A(V_n))_{n\in\IN}]$ with respect to the norm
$$\|T\|=\inf\left\{\sup_{n\in\IN}\|T^{(n)}+S^{(n)}\|_{E_n}: (S^{(n)})_{n\in\IN}\in I_G(\A)\right\},$$
where the norm of each operator $T^{(n)}$ is given by the $*$-representation on $E_n$ as above. This norm is called the \emph{reduced norm}.

Similarly, the \emph{maximal twisted Roe algebra at infinity} $C^*_{u,{\rm max},\infty}((X_n,\A(V_n))_{n\in\IN})$ is defined to be the completion of $\ICu[(X_n,\A(V_n))_{n\in\IN}]$ with respect to the norm
$$\|T\|_{\rm max}=\sup\left\{\|\phi(T)\|: \phi: \ICu[(X_n,\A(V_n))_{n\in\IN}]\to\B(\H_{\phi})\text{ is a $*$-representation}\right\}.$$
This norm is called the \emph{maximal norm}.
\end{Def}

There is a canonical quotient map
\begin{equation}\label{twisted quotient map}\pi^{\A}:C^*_{u,{\rm max},\infty}((X_n,\A(V_n))_{n\in\IN})\to\Cau((X_n,\A(V_n))_{n\in\IN}).\end{equation}
The following proposition can be seen as an "at infinity" version of \cite[Proposition 3.7]{JR2013}.

\begin{Pro}\label{twisted K-amenable at infinity for FCE spaces}
The canonical quotient $\pi^{\A}$ in \eqref{twisted quotient map} is a $*$-isomorphism.
\end{Pro}

The proof strategy involves decomposing the twisted algebras at infinity into several smaller subalgebras, each of which possesses a unique $C^*$-norm. We then proceed to glue these subalgebras together, and the isomorphism can be established by applying five lemma. To do this, we shall first discuss ideals of the twisted algebras supported on certain open subsets of $\IR_+\times V_n$.

\begin{Def}\label{coherent system for FCE}
A collection $O=(O_{x,n})_{x\in X_n, n\in\IN}$ of open subsets of $\IR_+\times V_n$, is said to be a \emph{coherent system} if for all but finitely many $n\in\IN$, there exists $r>0$ such that the following conditions hold:\begin{itemize}
\item[(1)]for any $x,x'\in X_n$, and $C\subseteq B(x,l_n)\cap B(x',l_n)$, we have
$$O_{x,n}\cap B_{\IR_+\times W_C(x)}(t_x(y)(s(y)),r)=t_{xx'}(O_{x',n}\cap B_{\IR_+\times W_C(x')}(t_{x'}(y)(s(y)),r))$$
for any $y\in B_{X_n}(x,l_n)\cap B_{X_n}(x',l_n)$ where
$$W_C(x)=\textup{affine-span}\{t_x(w)(s(w))\mid w\in C\subseteq B(x,l_n)\cap B(x',l_n)\}$$
$$W_C(x')=\textup{affine-span}\{t_{x'}(w)(s(w))\mid w\in C\subseteq B(x,l_n)\cap B(x',l_n)\};$$
\item[(2)]for any $x\in X_n$, $n\in\IN$ and $k>0$, we have
$$\overline{O_{x,n}\cap(\IR_+\times W_k(x))}^{\beta_{W,W_k(x)}}= O_{x,n}\cap(\IR_+\times V_n)$$
where for $W_k(x)\subseteq V_n$ with $W'_k(x)=V_n\ominus W_k(x)$ and $v\in V_n$, there is a unique decomposition $v=(w',w)=W_k'(x)\oplus W_k(x)$ and
$$\overline{O}^{\beta_{W,W_k(x)}}=\left\{(\tau,v)\in\IR_+\times V_n\mid\left(\sqrt{\tau^2+\|w'\|^2},w\right)\in O\right\}.$$
\end{itemize}\end{Def}

For given $r>0$, set $O_{x,n}=\bigcup_{y\in B(x,l_n)}B(t_x(y)(s(y)),r)$ for any $x\in X_n$ and $n\in\IN$. It is easy to see such a collection $(O_{x,n})_{x\in X_n,n\in\IN}$ forms a typical example of a coherent system. For any two coherent collections $O^{(1)}$ and $O^{(2)}$, we say $O^{(1)}\subseteq O^{(2)}$ if $O^{(1)}_{x,n}\subseteq O^{(2)}_{x,n}$ for all $x\in X_n$ and $n\in\IN$. Denote $O^{(1)}\cup O^{(2)} = (O^{(1)}_{x,n}\cup O^{(2)}_{x,n})_{x\in X_n,n\in\IN}$ and $O^{(1)}\cap O^{(2)} = (O^{(1)}_{x,n}\cap O^{(2)}_{x,n})_{x\in X_n,n\in\IN}$. Note that if $O^{(1)}$ and $ O^{(2)}$ are coherent, so are both $O^{(1)}\cup O^{(2)}$ and $O^{(1)}\cap O^{(2)}$.

\begin{Def}
Let $O=(O_{x,n})_{x\in X_n,n\in\IN}$ be a coherent system of open subsets of $\IR_+\times V_n$, $n\in\IN$. Define $\ICu[(X_n,\A(V_n))_{n\in\IN}]_{O}$ to be the $*$-subalgebra of $\ICu[(X_n,\A(V_n))_{n\in\IN}]$ generated by the equivalence classes of the sequences $[(T^{(0)},\cdots,T^{(n)},\cdots)]$ such that
$$\supp(T^{(n)}(x,y))\subseteq O_{x,n}$$
for all $x,y\in X_{n}$ and $n\geq N$ for some $N\in\IN$ large enough depending on the propagation of the sequence $[(T^{(0)},\cdots,T^{(n)},\cdots)]$.

Define $\Cau((X_n,\A(V_n))_{n\in\IN})_O$ (the maximal version $C^*_{u,{\rm max},\infty}((X_n,\A(V_n))_{n\in\IN})_O$, respectively) to be norm closure of $\ICu[(X_n,\A(V_n))_{n\in\IN}]_O$ under the norm reduced (maximal, respectively) norm.
\end{Def}

We would like to point out that the maximal norm may not equal the supremum norm of all representations of $\ICu[(X_n,\A(V_n))_{n\in\IN}]_O$. 

\begin{Def}\label{separate}
Let $\Gamma_n$ be a subset of $X_n$ for each $n\in\IN$ and denote $\Gamma=(\Gamma_n)_{n\in\IN}$ and $r>0$. A coherent system $O=(O_{x,n})_{x\in X_n,n\in\IN}$ of open subsets of $(\IR_+\times V_n)_{n\in\IN}$ is said to be \emph{$(\Gamma,r)$-separate} if there exist open subsets $(O_{x,n,\gamma})_{\gamma\in\Gamma_n\cap B(x,l_n)}$ of $\IR_+\times V_n$ for all $x\in X_n$, $n\in\IN$, such that\begin{itemize}
\item[(1)]$O_{x,n}=\bigcup_{\gamma\in\Gamma_n\cap B(x,l_n)}O_{x,n,\gamma}$;
\item[(2)]$O_{x,n,\gamma}\cap O_{x,n,\gamma'}=\emptyset$ for distinct $\gamma, \gamma'\in\Gamma_n\cap  B(x,l_n)$;
\item[(3)]$O_{x,n,\gamma}\subseteq B(t_x(\gamma)(s(\gamma)),r)$ for each $\gamma\in \Gamma_n\cap B_{Q_n}(q,l_n)$.
\end{itemize}\end{Def}

Let $O$ be a $(\Ga,r)$-separated coherent system. For any $R>0$, we write $Y_{\gamma,R}$ for $B_{X_n}(\gamma,R)$ for any $\gamma\in \Ga_n$ for brevity. Note that the collection $Y=\{Y_{\gamma,R}\}_{\gamma\in\Ga}$ is uniformly bounded. Additionally, $(X_n)_{n\in\IN}$ has uniformly bounded geometry. Therefore, there exists a positive integer $N$ such that $\#Y_{\gamma,R}\leq N$.

For an open subset $U\subseteq V_n$, we define $\A_n(U)$ to be the ideal of $\A(V_n)$ generated by all functions whose supports are contained within $U$. To study the algebraic structure of $\ICu[(X_n,\A(V_n))_{n\in\IN}]_O$, we still need the following algebras. 
\begin{Def}\label{algebra A-infty}
Let $\IA_{\infty}[(Y_{\gamma,R}:\gamma\in\Ga_n)_{n\in\IN}]$ to be the subalgebra of
\begin{equation}\label{denomenator of A-infty}
\frac{\prod_{n\in\IN}(\bigoplus_{\gamma\in\Ga_n}\IC[Y_{\gamma,R}]\wox\A(O_{\gamma,n,\gamma}))}{\bigoplus_{n\in\IN}(\bigoplus_{\gamma\in\Ga_n}\IC[Y_{\gamma,R}]\wox\A(O_{\gamma,n,\gamma}))}
\end{equation}
consisting of all elements of the form $[(T^{(0)},\cdots,T^{(n)},\cdots)]$ where
$$T^{(n)}=\bigoplus_{\gamma\in\Gamma_n}T_{\gamma}^{(n)}$$
with
$$T^{(n)}_{\gamma}\in \IC[Y_{\gamma,R}]\wox\A(O_{\gamma, n,\gamma})$$
and the family $(T^{(n)}_{\gamma})_{\gamma\in\Ga_n,n\in\IN}$ satisfies the conditions in Definition \ref{twisted Roe algebra for FCE} with uniform constants. In particular, there exists $k>0$ such that $T^{(n)}_{\gamma}(x,y)$ is in the image of $\beta_{W_k(x),V_n}$ for any $n\in\IN$.
\end{Def}

\begin{Lem}\label{ICu is isomorphic to IA}
Let $O$ be a $(\Ga,r)$-separate coherent system for some $\Ga=(\Ga_n)_{n\in\IN}$ and $r>0$ as above. Then there exists a canonical isomorphism
\[\ICu[(X_n,\A(V_n))_{n\in\IN}]_O\cong\lim_{R\to\infty}\IA_{\infty}[(Y_{\gamma,R}:\gamma\in\Ga_n)_{n\in\IN}].\]
\end{Lem}

\begin{proof}
Since the sequence $\{Y_{\gamma,R}\}_{\gamma\in\Ga}$ is uniformly bounded, all ghost elements in $\prod^u_{\gamma\in\Ga}C^*(Y_{\gamma,R})$ must be compact. The rest of the proof is the same with that of \cite[Lemma 6.12]{CWY2013}.
\end{proof}

\begin{Lem}\label{separated case: twisted K-amenable at infinity}
Let $O$ be a $(\Ga,r)$-separate coherent system as in Definition \ref{separate}. Then the maximal norm on $\ICu[(X_n,\A(V_n))_{n\in\IN}]_O$ is equal to the reduced norm.
\end{Lem}

\begin{proof}
By Lemma \ref{ICu is isomorphic to IA}, it suffices to show that the maximal norm and the reduced norm coincide on the algebra $\IA_{\infty}[(Y_{\gamma,R}:\gamma\in\Ga_n)_{n\in\IN}]$.

For fixed $R>0$, there exists $N>0$ such that $\#Y_{\gamma,R}\leq N$ for any $\gamma\in\Ga_n$ and $n\in\IN$. Therefore, $\IC[Y_{\gamma,R}]$ can be viewed as a subalgebra of $M_N(\IC)\wox\K$. Then, $\prod_{n\in\IN}(\bigoplus_{\gamma\in\Ga_n}\IC[Y_{\gamma,R}]\wox\A(O_{\gamma,n,\gamma}))$ is $*$-isomorphic to a subalgebra of
\[\begin{split}\prod_{\gamma\in\Ga_n,n\in\IN}M_N(\IC)\wox\K\wox\A(O_{\gamma, n, \gamma})&\cong \prod_{\gamma\in\Ga_n,n\in\IN}M_N\left(\K\wox\A(O_{\gamma, n, \gamma})\right)\\
&\cong\prod_{\gamma\in\Ga_n,n\in\IN} M_N\left(\K\wox C_0(O_{\gamma, n, \gamma})\wox\Cl(V_n)\right).\end{split}\]
The maximal and the reduced norm induce two norms on each $M_N(\K\wox C_0(O_{\gamma, n, \gamma})\wox\Cl(V_n))$ for any $\gamma\in\Ga_n$ and $n\in\IN$, and the norm on the direct product is given by the supremum of the sequence. Notice that $M_N(\IC)$, $\Cl(V_n)$ are finite dimensional and $C_0(O_{\gamma, n, \gamma})$ is commutative, thus nuclear. As a result, there is a unique norm on
$$M_N\left(\K\wox C_0(O_{\gamma, n, \gamma})\wox\Cl(V_n)\right).$$
Thus, the maximal and the reduced norm coincides on each $\IC[Y_{\gamma,R}]\wox\A(O_{\gamma,n,\gamma})$. This completes the proof.
\end{proof}

Now, we are ready to prove Proposition \ref{twisted K-amenable at infinity for FCE spaces}.

\begin{proof}[Proof of Proposition \ref{twisted K-amenable at infinity for FCE spaces}]
Let $O_{x,n}(r)=\bigcup_{y\in B(x,l_n)}B(t_x(y)(s(y)),r)$. Then $O(r)=\{O_{x,n}(r)\}_{x\in X_n,n\in\IN}$ forms a coherent system. For any $r<r'$, there exists a canonical inclusion
$$\ICu[(X_n,\A(V_n))_{n\in\IN}]_{O(r)}\to\ICu[(X_n,\A(V_n))_{n\in\IN}]_{O(r')}.$$
By definition, we that
$$\ICu[(X_n,\A(V_n))_{n\in\IN}]\cong\lim_{r\to\infty}\ICu[(X_n,\A(V_n))_{n\in\IN}]_{O(r)}.$$
To show $\pi^{\A}$ in \eqref{twisted quotient map} is a $*$-isomorphism, it suffices to show that
$$\pi^{\A}_{O(r)}:C^*_{u,{\rm max},\infty}((X_n,\A(V_n))_{n\in\IN})_{O(r)}\to\Cau((X_n,\A(V_n))_{n\in\IN})_{O(r)}$$
is an isomorphism for any $r>0$.

For given $r>0$, since $(X_n)_{n\in\IN}$ has uniformly bounded geometry, there exists $N>0$ such that $\#B(x,r)<N$ for all $x\in X_n$, $n\in\IN$. It follows from \cite[Lemma 6.7]{Yu2000} that there exists an integer $J_r> 0$ independent of $n$ such that we have the following decompositions for all $n\in\IN$:\begin{itemize}
\item[$\bullet$]$X_n=\bigcup_{j=1}^{J_r}\Gamma_n^{j}$, where $\Gamma_n^{j}\subseteq X_n$;
\item[$\bullet$]$\Gamma_n^{(j)}\cap\Gamma_n^{(j')}=\emptyset$ whenever $j\ne j'$;
\item[$\bullet$]for any $x\in X_n$ and any distinct $\gamma,\gamma'\in\Gamma_n^{(j)}\cap B(x,l_n)$.
$$d(t_x(\gamma)(s(\gamma)),t_x(\gamma')(s(\gamma')))>2r$$
in $V_n$.\end{itemize}
For each $j\in\{1,2,\cdots,J_r\}$, let
$$O_{x,n}^{(j)}(r)=\bigcup_{\gamma\in\Gamma_n^{(j)}\cap B(x,l_n)}B_{\IR_+\times V_n}(t_{x}(\gamma)(s(\gamma)),r)$$
for all $x\in X_n$, $n\in\IN$. Then the system
$$O^{(j)}(r)=\left(O_{x,n}^{(j)}(r)\right)_{x\in X_n,n\in\IN}$$
is coherent for each $j\in\{1,2,\cdots,J_r\}$ and 
$$O(r)=\bigcup_{j=1}^{J_r}O^{(j)}(r).$$
Let $\Gamma^{(j)}=\left(\Gamma_n^{(j)}\right)_{n\in\IN}$. It is clear that the system $O^{(j)}(r)$ and
$$O^{(j)}(r)\cap O^{(j')}(r)$$
are $\left(\Gamma^{(j)},r\right)$-separate for any $j,j'\in\{1,2,\cdots,J_r\}$. By Lemma \ref{separated case: twisted K-amenable at infinity}, we have that
\begin{equation}\label{piAO isomorphism}\pi^{\A}_{O^{(j)}(r)}:C^*_{u,{\rm max},\infty}((X_n,\A(V_n))_{n\in\IN})_{O^{(j)}(r)}\to\Cau((X_n,\A(V_n))_{n\in\IN})_{O^{(j)}(r)}\end{equation}
and
\begin{equation}\label{piAOcap isomorphism}\pi^{\A}_{O^{(j)}(r)\cap O^{(j')}(r)}:C^*_{u,{\rm max},\infty}((X_n,\A(V_n))_{n\in\IN})_{O^{(j)}(r)\cap O^{(j')}(r)}\to\Cau((X_n,\A(V_n))_{n\in\IN})_{O^{(j)}(r)\cap O^{(j')}(r)}\end{equation}
are isomorphisms for any $j,j'\in\{1,\cdots,J_r\}$. By using a similar argument in \cite[Lemma 6.3]{Yu2000}, we have the following pushout diagram
\[\begin{tikzcd}
 \ICu[(X_n,\A(V_n))_{n\in\IN}]_{O^{(j)}(r)\cap O^{(j')}(r)}\arrow[r] \arrow[d] & \ICu[(X_n,\A(V_n))_{n\in\IN}]_{O^{(j')}(r)}\arrow[d]   \\
\ICu[(X_n,\A(V_n))_{n\in\IN}]_{O^{(j)}(r)}\arrow[r] & \ICu[(X_n,\A(V_n))_{n\in\IN}]_{O^{(j)}(r)\cup O^{(j')}(r)}\end{tikzcd}\]
In a pushout diagram, the $C^*$-norm on a pushout is uniquely determined by the norms on the other three $C^*$-algebras in the diagram.
Combining this fact with \eqref{piAO isomorphism} and \eqref{piAOcap isomorphism}. we have that $\pi^{\A}_{O^{(j)}(r)\cup O^{(j')}(r)}$ is an isomorphism. By the inductive argument, $\pi^{\A}_{O(r)}$ is an isomorphism. This finishes the proof.
\end{proof}

\subsection{Geometric Bott map and Dirac map}

In \cite{CWY2013}, X.~Chen, Q.~Wang and G.~Yu introduced a Bott map and a Dirac map for maximal Roe algebra at infinity, i.e.,
$$\beta_t:\S\wox C^*_{u,{\rm max},\infty}((X_n)_{n\in\IN})\to C^*_{u,{\rm max},\infty}((X_n,\A(V_n))_{n\in\IN})$$
and
$$\alpha_t:C^*_{u,{\rm max},\infty}((X_n,\A(V_n))_{n\in\IN})\to C^*_{u,{\rm max},\infty}((X_n,\K(L^2_n))_{n\in\IN}).$$
These two maps can still be defined for the reduced case. We shall recall some details in this section.

\subsubsection{The Bott map}
We shall first recall the Bott map for a finite dimensional Euclodean space. For any $n\in\IN$ and $v\in V_n$, the \emph{Clifford operator} is defined to be
$$C_{V_n,v}:V_n\to V_n^0\subseteq\Cl(V_n^0)$$
the unbounded function $w\mapsto w- v\in V_n^0\subseteq\Cl(V_n^0)$ for all $v\in V_n$. For each $n\in\IN$ and $x\in X_{n}$, the inclusion of the $0$-dimensional affine subspace $W_0(x)=\{t_{x}(x)(s(x))\}$ into $V_n$ induces a $*$-homomorphism
$$\beta(x):\S\cong\A\left(W_0(x)\right)\to\A(V_n)$$
by the formula
$$\left(\beta(x)\right)(g)=g\left(X\wox 1+1\wox C_{V_n,t_{x}(x)(s(x))}\right).$$

\begin{Def}[The Bott map]
For each $t\in [1,\infty)$, define a map
$$\beta_t:\S\wox\ICu[(X_n)_{n\in\IN}]\to\ICu[(X_n,\A(V_n))_{n\in\IN}]$$
by the formula
$$\beta_t(g\wox T)=[((\beta_t(g\wox T))^{(0)},\cdots,(\beta_t(g\wox T))^{(n)},\cdots)]$$
for all $g\in\S$, $T=[(T^{(0)},\cdots,T^{(n)},\cdots)]\in\ICu[(X_n)_{n\in\IN}]$, where
$$(\beta_t(g\wox T))^{(n)}(x,y)=\left(\beta(x)\right)(g_t)\wox T^{(n)}(x,y)$$
for $x,y\in X_{n}$, $n\in\IN$, and $g_t(r)=g(\frac rt)$ for all $r\in\IR$.
\end{Def}

Let $A,B$ be two $C^*$-algebras. Denoted by
$$\Q(B)=\frac{C_b(\IR_+,B)}{C_0(\IR_+,B)}.$$
Recall that a family of map $(\beta_t:A\to B)_{t\in\IR_+}$ is an asymptotic morphism if $(\beta_t)_{t\in\IR_+}$ induces a $*$-homomorphism
$$\beta_t:A\to\Q(B).$$
An asymptotic morphism also induces a homomorphism on $K$-theory, the reader is referred to \cite{GHT2000} for a detailed introduction to asymptotic morphisms.

\begin{Lem}\label{Bott for FCE}
The maps $(\beta_t)_{t\in\IR}$ extend respectively to asymptotic morphisms
$$\beta:\S\wox\Cau((X_n)_{n\in\IN})\leadsto\Cau((X_n,\A(V_n))_{n\in\IN}),$$
$$\beta_{max}:\S\wox C^*_{u,{\rm max},\infty}((X_n)_{n\in\IN})\leadsto C^*_{u,{\rm max},\infty}((X_n,\A(V_n))_{n\in\IN}).$$
\end{Lem}

\begin{proof}
The maximal case is essentially proved in \cite[Lemma 7.3]{CWY2013}. The reduced case is proved in \cite[Lemma 7.3]{GLWZ2022}, one can just replace the Banach space $B$ in \cite{GLWZ2022} with a Hilbert space. We provide a straightforward proof here for the case when $X$ is sparse (coarse disjoint union of finite spaces)

It is not hard to show that $\beta$ is asymptotically linear. We shall only show that
$$\|\beta_t(ST)-\beta_t(S)\beta_t(T)\|\to0\quad\mbox{as}\quad t\to\infty.$$
For fixed $R>0$, there exists $N_R>0$ such that $l_n>R$ for all $n>N_R$. Let $x,y\in X_{n}$ with $d(x,y)<R$ and $n\geq N_R$, then both $x$ and $y$ are in the regions of trivializations of $t_x$ and $t_y$. For each fixed $g\in\S$, it suffices to prove
$$\|\beta(x)(g_t)-(t_{xy})_*(\beta(y)(g_t))\|$$
tends to $0$ uniformly for all $x,y\in X_{n}$ with $n>N_R$ and $d(x,y)\leq R$, where $t_{xy}=t_{x}\circ t_{y}^{-1}$ is as in Definition \ref{FCE}. Thus it is direct to check that
$$(t_{xy})_*\left(C_{V_n,t_{y}(y)(s(y))}\right)=C_{V_n,t_{x}(y)(s(y))}.$$

For the generator $g=\frac{1}{x\pm i}$ of $\S$, by using a standard argument, one has that
\begin{equation}\begin{split}
&\|\beta(x)(g_t)-(t_{xy})_*\beta(y)(g_t)\|\\
=&\left\|g_t\left(X\wox 1+1\wox C_{V_n,t_{x}(x)(s(x))}\right)-(t_{xy})_*\left(g_t\left(X\wox 1+1\wox C_{V_n,t_{y}(y)(s(y))}\right)\right)\right\|\\
=&\left\|g_t\left(X\wox 1+1\wox C_{V_n,t_{x}(x)(s(x))}\right)-g_t\left(X\wox 1+1\wox C_{V_n,t_{x}(y)(s(y))}\right)\right\|\\
\leq&\frac 1t\|t_{x}(x)(s(x))-t_{x}(y)(s(y))\|\leq\frac 1t\rho_+(R).
\end{split}\end{equation}
It follows from an approximation argument as in \cite[Lemma 7.3]{Yu2000}, $(\beta_t)_{t\in\IR_+}$ is a well-defined $*$-homomorphism from $\S\wox\ICu[(X_n)_{n\in\IN}]$ to $\Q(\Cau((X_n,\A(V_n))_{n\in\IN}))$.

For each $n\in\IN$, as we can view $\S$ as $\A(W_0(x))$ for $x\in X_{n}$, we define $\beta^z_{V_n,x}:\S\to\A(V_n)$ is induced by the inclusion $i^z_x:\{t_{x}(z)(s(z))\}\to V_n$ which maps $t_{x}(z)(s(z))$ to itself. For every $g\in\S$, we define $\beta_{V_n,x}:\S\to C(X_n,\A(V_n)\wox\K)$ by
$$(\beta_{V_n,x}(g))(z)=\left(\beta_{V_n,x}^z(g)\right).$$
With the notation above, we define a bounded module homomorphism $N_g: E_n\to E_n$ for each $n\in\IN$ by
$$(N_g(f))(x)=\beta_{V_n}(g)\cdot f(x)$$
for all $f\in E_n$, $n\in\IN$, where $E_n$ is the Hilbert module defined as in \eqref{Hilbert module E_n}.  For any $T=[(T^{(n)})_{n\in\IN}]\in \ICu[(X_n)_{n\in\IN}]$, we have that
$$\beta_t(g\wox T^{(n)})=N_{g_t}\cdot(1\wox T^{(n)})$$
for all $g\in\S$ and sufficiently large $n\in\IN$, where 
$$\left((1\wox T^{(n)})f\right)(x)=\sum_{y\in X_{n}}(1\wox T^{(n)}(x,y))f(y)$$
for all $f\in E_n$.

From the argument above, it follows that
$$\|\beta_t(g\wox T^{(n)})\|\leq\|g\|\cdot\|T^{(n)}\|$$
for all $g\in\S$ and $T^{(n)}\in\IC[X_n]$. Moreover, the image of $I_G$ under the Bott map is clearly in $I_G(\A)$. Hence $\beta$ extends to a $C^*$-asymptotic morphism from $\S\otimes_{\rm max}\Cau((X_n)_{n\in\IN})$ to $\Cau((X_n,\A(V_n))_{n\in\IN})$. Since $\S$ is nuclear, we conclude that the lemma holds as desire.
\end{proof}

\subsubsection{The Dirac map}

Now, let us define the Dirac map. For any $n\in\IN$, set $E_n=\text{linear-span}\{V_n\}\subseteq\H$. Define
$$L^2_n=L^2\left(E_n,\Cl(E_n)\right)$$
to be the graded infinite dimensional complex Hilbert space of square-integrable $\Cl(E_n)$-valued functions on $E_n$, where $E_n$ is endowed with the Lebesgue measure. The grading on $L^2_n$ is inherited from the grading of the Clifford algebra. Choose an orthonormal basis $\{e_1,\cdots, e_n\}$ for $E_n$, and let $x_1,\cdots, x_n$ be the corresponding coordinates. The \emph{Dirac operator}, denoted by $D_{V_n}$, is defined by
$$(D_{V_n}u)(x)=\sum_{i=1}^n(-1)^{deg(u)}\frac{\partial u}{\partial x_i}(x)\cdot e_i,$$
for any $u\in\S(E_n)\subseteq L^2_n$, where $\S(E_n)$ is the subspace of Schwartz functions in $L^2_n$. Denote by $\K(L^2_n)$ the graded $C^*$-algebra of all compact operators on $L^2_n$.

Let $t:V_n\to V_n$ be an affine isometric bijection. Then it induces a unitary operator on $L^2_n$, which further induces a $*$-isomorphism
$$t_*:\K(L^2_n)\to\K(L^2_n)$$
by using a conjugation with the unitary. For any $n\in\IN$ and $x,z\in X_n$ with $C=B(x,l_n)\cap B(z,l_n)\neq \emptyset$, we have a bijective isometry
$$t_{xz}:W_C(z)\to W_C(x).$$
We can extend this isometry to a unitary on $E_n$ by choosing a unitary operator
$$U_{xz}: W_C(z)^{\bot}\to W_C(x)^{\bot},$$
where $W_C(z)^{\bot}=E_n\ominus W_C(z)$. Then $U_{xz}\oplus t_{xz}:E_n\to E_n$ forms an affine isometry from $E_n$ to $E_n$. We still denote it by $(t_{xz})_*:\K(L^2_n)\to \K(L^2_n)$ the $C^*$-homomorphism induced by $U_{xz}\oplus t_{xz}$. For each pair $x,z$ as above, we choose a unitary operator $U_{xz}$ such that these unitary operators satisfy the cocycle condition: $U_{xz}=U_{xy}\circ U_{yz}$ for any $y\in B(x,l_n)\cap B(z,l_n)$.

\begin{Def}
Define $\ICu[(X_n,\S\K(L^2_n))_{n\in\IN}]$ to be the set of equivalent classes $T=[(T^{(n)})_{n\in\IN}]$ of sequences $(T^{(n)})_{n\in\IN}$ such that\begin{itemize}
\item[(1)] $T^{(n)}$ is a bounded function from $X_n\times X_n$ to $\S\wox\K(L^2_n)\wox\K$ for each $n\in\IN$ such that the sequence $(T^{(n)})_{n\in\IN}$ is uniformly bounded;
\item[(2)] there exists $R>0$ such that $T^{(n)}(x,y)=0$ whenever $x,y\in X_n$ satisfy that $d(x,y)>R$ and $n\in\IN$;
\end{itemize}
The equivalence relation $\sim$ on these sequences is defined by
$$\left(T^{(n)}\right)_{n\in\IN}\sim \left(S^{(n)}\right)_{n\in\IN}$$
if and only if
$$\lim_{n\to\infty}\sup_{x,y\in Z_{d,n}}\|T^{(n)}(x,y)-S^{(n)}(x,y)\|_{\S\K(L^2_n)}=0.$$
\end{Def}

The algebraic structure of $\ICu[(X_n,\S\K(L^2_n))_{n\in\IN}]$ is defined similar with Definition \ref{twisted Roe algebra for FCE} (also see \cite[Definition 7.6]{CWY2013}). The product structure for $\ICu[(X_n,\S\K(L^2_n))_{n\in\IN}]$ is defined as follows. For any two elements $T=[(T^{(0)},\cdots,T^{(n)},\cdots)]$ and $S=[(S^{(0)},\cdots,S^{(n)},\cdots)]$ in $\ICu[(X_n,\S\K(L^2_n))_{n\in\IN}]$, the product is defined as
$$TS=[((TS)^{(0)},\cdots,(TS)^{(n)},\cdots)]$$
where there exists a sufficiently large $N\in\IN$ which depends on the propagation of $T$, such that $(TS)^{(n)}=0$ for $n < N$ and
$$(TS)^{(n)}(x,y)=\sum_{z\in X_{n}}\left(T^{(n)}(x,z)\right)\cdot\left((t_{xz})_*(S^{(n)}(z,y))\right),$$
where $(t_{xz})_*:\K(L^2_n)\to\K(L^2_n)$ is defined as above.

The $*$-structure for $\ICu[(X_n,\S\K(L^2_n))_{n\in\IN}]$ is defined by
$$[(T^{(0)},\cdots,T^{(n)},\cdots)]^*=[((T^*)^{(0)},\cdots,(T^*)^{(n)},\cdots)]$$
where
$$(T^*)^{(n)}(x,y)=(t_{xy})_*\left(\left(T^{(n)}(y,x)\right)^*\right)$$
for all but finitely many $n$, and $0$ otherwise. Then $\ICu[\PdGA]$ is made into a $*$-algebra with the additional usual matrix operations. 

The reduced norm on $\ICu[(X_n,\S\K(L^2_n))_{n\in\IN}]$ is defined similar with Definition \ref{twisted Roe algebra for FCE}. Define
$$H_n=\bigoplus_{x\in X_n}\ell^2\left(B(x,\frac{l_n}2)\right)\wox L^2_n\ox \S\ox\K.$$
For an element $[(T^{(n)})_{n\in\IN}]\in\ICu[(X_n,\S\K(L^2_n))_{n\in\IN}]$, we fix a representation sequence $(T^{(n)})_{n\in\IN}$. The action of $T^{(n)}$ on $H_n$ is defined by
$$T_n(\delta_x\ox\delta_y\ox (v\ox s\ox k))=\sum_{z\in X_n}\delta_z\ox\delta_y\ox (t_{yz})_*(T_n(z,x)) (v\ox s\ox k),$$
for any $x,y\in X_n$, $v\in L^2_n$ and $s\ox k\in\S\wox\K$.
Follwoing a similar argument in Remark \ref{reduced norm is well-defined}, one can check that this action is compatible with the algebraic structure of $\ICu[(X_n,\S\K(L^2_n))_{n\in\IN}]$.

We can similarly define the ghost ideal $I_G(\S\K)$ in this situation.
Define $B^*((X_n,\S\K(L^2_n))_{n\in\IN})$ to be the completion of $\ICu[(X_n,\S\A(L^2_n))_{n\in\IN}]$ with respect to the norm
$$\|T\|=\limsup_{n\to\infty}\|T^{(n)}\|_{H_n},$$
where the norm of each operator $T^{(n)}$ is given by the $*$-representation on $H_n$ as above. An element in $[(T^{(n)})_{n\in\IN}]\in B^*((X_n,\S\K(L^2_n))_{n\in\IN})$ is a ghost if for any $R>0$, 
$$\lim_{n\to\infty}\sup_{x,y\in X_n}\|\chi_{B(x,R)}T^{(n)}\chi_{B(y,R)}\|=0.$$
Denoted by $I_G(\S\K)$ to be the ideal of all ghost elements in $B^*((X_n,\S\K(L^2_n))_{n\in\IN})$.

\begin{Def}
Define $C^*_{u,\infty}((X_n,\S\K(L^2_n))_{n\in\IN})$ to be the completion of $\ICu[(X_n,\S\K(L^2_n))_{n\in\IN}]$ with respect to the norm
$$\|T\|=\inf\left\{\sup_{n\in\IN}\|T^{(n)}+S^{(n)}\|_{H_n}\mid (S^{(n)})_{n\in\IN}\in I_G(\S\K)\right\}.$$

Define $C^*_{u,{\rm max},\infty}((X_n,\S\K(L^2_n))_{n\in\IN})$ to be the completion of $\ICu[(X_n,\S\K(L^2_n))_{n\in\IN}]$ with respect to the norm
$$\|T\|_{\rm max}=\sup\left\{\|\phi(T)\|\mid \phi: \ICu[(X_n,\S\K(L^2_n))_{n\in\IN}]\to\B(\H_{\phi})\text{ is a $*$-representation}\right\}.$$
\end{Def}

For an affine space $V\subseteq E_n$, we denote by $V^{\bot}=E_n\ominus V$ to be the orthogonal complement of $V$ in $E_n$. We define the Bott-Dirac operator on the linear space $V^{\bot}$ to be
$$B_{V^{\bot}}=C_{V^{\bot}}+D_{V^{\bot}},$$
where $C_{V^{\bot}}$ is the Clifford operator associated with $0$ and $D_{V^{\bot}}$ is the Dirac operator on $V^{\bot}$. It is proved in \cite[Appendix D.3]{HIT2020} that the eigenvalue of $B^2_{V^{\bot}}$ is all non-negative even numbers with finite-dimensional eigenspaces. The eigenspace of $0$ is the one-dimensional linear subspace spanned by the Gauss function on $V^{\bot}$.

\begin{Def}[The Dirac map]
For each $t\in [1, \infty)$, define a map
$$\alpha_t: \ICu[(X_n,\A(V_n))_{n\in\IN}]\to\ICu[(X_n,\S\K(L^2_n))_{n\in\IN}]$$
by the formula
$$\alpha_t(T)=[(\alpha_t(T))^{(0)},\cdots,(\alpha_t(T))^{(n)},\cdots]$$
for $T =[(T^{(0)},\cdots,T^{(n)},\cdots)]\in\ICu[(X_n,\A(V_n))_{n\in\IN}]$, with
$$(\alpha_t(T))^{(n)}(x,y)=(\theta_t^k(x))\left(T_1^{(n)}(x,y)\right)$$
for any $x,y\in X_n$, $n\in\IN$, where\begin{itemize}
\item the number $k>0$ (independent of $n$) and $T_1(x,y)$ is given as in condition (2) and condition (6) of Definition \ref{twisted Roe algebra for FCE}.
\item the map
$$\theta_t^k(x):\A(W_k(x))\wox\K\to\K(L^2_n)\wox\K$$
is defined by the formula
$$\left(\theta_t^k(x)\right)(g\ox h\ox k)=g_t(B_{W_k(x)^{\bot}}\wox 1+1\wox D_{W_k(x)})(1\wox M_{h_t})\wox k$$
for all $g\in\S$, $h\in\C(W_k(x))$ and $k\in\K$.
\end{itemize}\end{Def}

The map $\theta_t^k$ is defined as an analogue of \cite[Definition 2.8]{HKT1998}. By using Rellich Lemma and the ellipticity of the Dirac operator, one can check the map $\alpha_t$ is well-defined.

\begin{Lem}\label{Dirac for FCE}
The maps $(\alpha_t)_{t\geq 1}$ extend respectively to asymptotic morphisms
$$\alpha:\Cau((X_n,\A(V_n))_{n\in\IN})\leadsto\Cau((X_n,\S\K(L^2_n))_{n\in\IN}),$$
$$\alpha_{\rm max}:C^*_{u,{\rm max},\infty}((X_n,\A(V_n))_{n\in\IN})\leadsto C^*_{u,{\rm max},\infty}((X_n,\S\K(L^2_n))_{n\in\IN}).$$
\end{Lem}

\begin{proof}
The maximal case is proved in \cite[Lemma 7.10]{CWY2013}. Following from the same argument, one can check the maps $(\alpha_t)_{t\geq 1}$ define a $*$-homomorphism
$$\alpha:\ICu[(X_n,\A(V_n))_{n\in\IN}]\to\Q(\Cau((X_n,\S\K(L^2_n))_{n\in\IN})).$$
By using Proposition \ref{twisted K-amenable at infinity for FCE spaces}, this map extends to a $C^*$-homomorphism by using the universal property.
\end{proof}

\subsubsection{A geometric Dirac-dual-Dirac construction}

Combining Lemma \ref{Bott for FCE} with Lemma \ref{Dirac for FCE}, we have the following commutative diagram on the level of $K$-theory.

\begin{equation}\label{Dirac-dual-Dirac for FCE}\begin{tikzcd}
K_{*+1}(C^*_{u,{\rm max},\infty}((X_n)_{n\in\IN}))\arrow[r,"\pi_*"] \arrow[d,"(\beta_{\rm max})_*"'] & K_{*+1}(\Cau((X_n)_{n\in\IN}))\arrow[d,"\beta_*"]   \\
K_*(C^*_{u,{\rm max},\infty}((X_n,\A(V_n))_{n\in\IN}))\arrow[r,"\pi_*^{\A}"] \arrow[d,"(\alpha_{\rm max})_*"'] & K_*(\Cau((X_n,\A(V_n))_{n\in\IN}))\arrow[d,"\alpha_*"]   \\ 
K_*(C^*_{u,{\rm max},\infty}((X_n,\S\K(L^2_n))_{n\in\IN}))\arrow[r,"\pi_*"] & K_*(\Cau((X_n,\S\K(L^2_n))_{n\in\IN}))
\end{tikzcd}\end{equation}

Following the argument in \cite[Theorem 7.11]{CWY2013} for the reduced case, we obtain the following:

\begin{Lem}\label{composition of bott and dirac}
The composition $\alpha_*\circ\beta_*$ and $(\alpha_{\rm max})_*\circ(\beta_{\rm max})_*$ are the identity maps.\qed
\end{Lem}

Now, we are ready to prove Theorem \ref{K-amenable at infinity for FCE spaces}.

\begin{proof}[Proof of Theorem \ref{K-amenable at infinity for FCE spaces}]
By Lemma \ref{composition of bott and dirac} and the diagram \eqref{Dirac-dual-Dirac for FCE}, the lemma follows from diagram-chasing.
\end{proof}

As a consequence of Theorem \ref{K-amenable at infinity for FCE spaces} and \cite[Theorem 1.1]{CWY2013}, we have the following result.

\begin{Cor}\label{CBC at infinity for FCE}
Let $(X_n)_{n\in\IN}$ be a sequence of finite metric spaces whose coarse disjoint union admits a fibred coarse embedding into Hilbert space, then the coarse assembly map at infinity for $(X_n)_{n\in\IN}$ is an isomorphism, i.e.,
$$\mu_{\infty}:\lim_{d\to\infty}K_*(\CauL((P_d(X_n))_{n\in\IN}))\to\lim_{d\to\infty}K_*(\Cau((P_d(X_n))_{n\in\IN}))$$
is an isomorphism.
\end{Cor}

\begin{proof}
It is direct to check that the Roe algebra at infinity is coarsely invariant as \cite[Theorem 5.1.15]{HIT2020}. Thus the left-hand side is equal to $K_*(\Cau((X_n)_{n\in\IN}))$, which is isomorphic to $K_*(C^*_{u,{\rm max},\infty}((X_n)_{n\in\IN}))$ by Theorem \ref{K-amenable at infinity for FCE spaces}. On the other side, the $K$-theory of both maximal and reduced localization algebra are isomorphic to
$$\frac{\prod_{n\in\IN}K_*(P_d(X_n))}{\bigoplus_{n\in\IN}K_*(P_d(X_n))},$$
since the localization algebras are local, see \cite[Proposition 4.3]{CWY2013} and \cite[Lemma 5.1]{GLWZ2022}. Then this corollary follows directly from \cite[Theorem 1.1]{CWY2013}.
\end{proof}

\subsection{A remark on boundary groupoids and a-T-menability}\label{subsec: boundary groupoid}

Besides the method with the aid of twisted algebras, there is another approach to study fibred coarse embeddings using the \emph{boundary coarse groupoid}. Let $X$ be a metric space with bounded geometry. For any $R>0$, define
$$\Delta_{R}=\{(x,y)\in X\times X\mid d(x,y)\leq R\}.$$
The \emph{coarse groupoid} of $X$, introduced in \cite{STY2002}, is defined to be
$$G(X)=\bigcup_{R\geq 0}\overline{\Delta_R}\subseteq\beta(X\times X),$$
where $\beta(X\times X)$ means the Stone-\v{C}ech compactification of $X\times X$. Actually, the coarse groupoid is principal, one can also take closure in $\beta(X)\times\beta(X)$. The unit space of $G(X)$ is $\beta(X)$ and the source (resp. range) map is extended from projection $p:X\times X\to X$ to the second (resp. first) coordinate. More details can be found in \cite{STY2002} and \cite[Chapter 10]{Roe2003}.

Note that $X$ is an open subset of $\beta(X)$, and it is also an invariant subset in $G(X)$. Denoted by $\partial_\beta X=\beta X\backslash X$ the Stone-\v{C}ech Corona. The \emph{boundary groupoid}, denoted by $G_{\infty}(X)$, is defined to be the restriction of $G(X)$ on the closed invariant subset $\partial_{\beta}X\subseteq\beta(X)$.

In the rest of this section, we assume that $X=\bigsqcup_{n\in\IN}X_n$ is a sparse space. The \emph{uniform Roe algebra at infinity}, denoted by $\IC_{u,\infty,\IC}[(X_n)_{n\in\IN}]$, is defined similarly to the Roe algebra at infinity but replacing the $\K(H_0)$-valued functions by $\IC$ valued function. More precisely, an element in$\IC_{u,\infty,\IC}[(X_n)_{n\in\IN}]$ is the equivalent class of a sequence
$$(T^{(n)}:X_n\times X_n\to\IC)_{n\in\IN}$$
such that $(T^{n})_{n\in\IN}$ is uniformly bounded and the propagation of the sequence $(T^{n})_{n\in\IN}$ is finite. Define $C^*_{u,\infty,\IC}((X_n)_{n\in\IN})$ to be the completion of $\IC_{u,\infty,\IC}[(X_n)_{n\in\IN}]$ under the norm defined by
$$\|T\|=\inf\left\{\sup_{n\in\IN}\|T^{(n)}+S^{(n)}\|\mid (S^{(n)})_{n\in\IN}\in I_{u,G}\right\},$$
for any $T=[(T^{n})_{n\in\IN}]$, where $I_{u,G}$ is the ghost ideal in $\prod_{n\in\IN}^uC^*_u(X_n)$.

\begin{Thm}\label{boundary groupoid algebra and Roe algebra at infinity}
The groupoid $C^*$-algebra $C^*_r(G_{\infty}(X))$ is $*$-isomorphic to $C^*_{u,\infty,\IC}((X_n)_{n\in\IN})$.
\end{Thm}

\begin{proof}
For any $f\in C_c(G_{\infty}(X))$, there exists $R>0$ such that $\supp(f)\subseteq \partial\Delta_R=\overline{\Delta_R}\backslash \Delta_R$. Since we have the exact short sequence in the sense of $*$-algebras (see \cite[Appendix C]{SpaWil2017})
$$0\to C_c(X\times X)\to C_c(G(X))\to C_c(G_{\infty}(X))\to 0,$$
we can extend $f$ to a continuous function $\wt f\in C_c(G(X))$. Without loss of generality, assume that the support of $\wt f$ is also contained in $\overline{\Delta_R}$. Define $T_f^{(n)}$ to be the restriction of $\wt f$ to $(X_n\times X_n)\cap\Delta_R$. Then the equivalent class of the sequence $(T_f^{(n)})_{n\in\IN}$ defines an element in $\IC_{u,\infty,\IC}[(X_n)_{n\in\IN}]$. Notice that the equivalent class of $(T_f^{(n)})_{n\in\IN}$ does not depend on the choice of $\wt f$. Indeed, if $g\in C_c(G(X))$ is another choice of the extension of $f$, then $g-\wt f\in C_c(X\times X)$ by using the short exact sequence above. There exists a sufficiently large $N>0$ such that $\wt f(x,y)=g(x,y)$ for all $(x,y)\in X_n\times X_n$ with $n>N$. We denote this map by
$$\Lambda: C_c(G_{\infty}(X))\to \IC_{u,\infty,\IC}[(X_n)_{n\in\IN}],\qquad f\mapsto [(T_f^{(n)})_{n\in\IN}].$$

On the other hand, for any $T=[T^{(n)}]\in\IC_{u,\infty,\IC}[(X_n)_{n\in\IN}]$, fix $(T^{(n)})_{n\in\IN}$ as a representation element with propagation equal to $R$. Then $f:\Delta_R\to\IC$ defined by
$$f_T(x,y)=\left\{\begin{aligned} &T^{(n)}(x,y),&&(x,y)\in X_n\times X_n\text{ for some }n;\\&0,&&\text{otherwise.}\end{aligned} \right.$$
Then $f_T$ extends to a continuous function on $\overline{\Delta_R}$, which restricts to an element $f_{T,\infty}$ in $C_c(G_{\infty}(X))$. For any different representation element $(S^{(n)})_{n\in\IN}$, since
$$\limsup_{n\to\infty}|S^{(n)}(x,y)-T^{(n)}(x,y)|\to 0,$$
thus we have that $f_{T,\infty}=f_{S,\infty}$ on $G_{\infty}(X)$. Define
$$\Psi:\IC_{u,\infty,\IC}[(X_n)_{n\in\IN}]\to C_c(G_{\infty}(X)),\quad\text{by}\quad [(T^{(n)})_{n\in\IN}]\mapsto f_{T,\infty}.$$

It is not hard to check that $\Lambda$ and $\Psi$ are $*$-homomorphisms by using a similar proof with \cite[Proposition 10.28]{Roe2003} and inverse maps to each other, thus are $*$-isomorphisms. It suffices to show their norm are the same. By \cite[Proposition 10.29]{Roe2003}, the \emph{uniform Roe algebra} of $C^*_u(X)$ is isomorphic to $C^*_r(G(X))$. For any $T=(T^{(n)})_{n\in\IN}\in\IC_{u,\infty,\IC}[(X_n)_{n\in\IN}]$, the representation sequence $(T^{(n)})_{n\in\IN}$ defines an element in the uniform Roe algebra $C^*_u(X)$, where $X=\bigsqcup_{n\in\IN}X_n$. Thus for any $N>0$, the map
$$C^*_u(X)\to C^*_r(G(X))\to C^*_r(G_{\infty}(X))$$
defined by
$$T_{>N}=(T^{(n)})_{n>N}\mapsto f_{T_{>N}}\mapsto f_{T_{>N},\infty}$$
is contractive. Notice that for any $N,M>0$, the image $f_{T_{>N},\infty}$ are equal $f_{T,\infty}$. Passing to the limit, one can descend this map to
$$\Psi:C^*_{u,\infty,\IC}((X_n)_{n\in\IN})\to C^*_r(G_{\infty}(X)),$$
which coincides with $\Psi$ defined as above.

Assume that $a\in C^*_{u,\infty,\IC}((X_n)_{n\in\IN})$ satisfies that $\Psi(a)=0$. Since the canonical quotient introduced in Section \ref{subsec: CNinfty}
$$\Phi:C^*_u(X)\to C^*_{u,\infty,\IC}((X_n)_{n\in\IN})$$
is surjective, there exists a sequence $(T^{(n)}\in C^*_u(X_n))_{n\in\IN}$ such that the direct sum of $T^{(n)}$ are in $C^*_u(X)$ and
$$\Phi\left(\bigoplus_{n\in\IN}T^{(n)}\right)=a.$$
By  \cite[Theorem 34]{FS2014} or \cite[Equation (5.3)]{WZ2023}, we have the following short exact sequence
$$0\to I_{u,G}\to C^*_r(G(X))\to C^*_r(G_{\infty}(X))\to 0,$$
where $I_{u,G}$ is the ghost ideal of $C^*_u(X)$. This means that $\sup_{x,y\in X_n}|T^{(n)}(x,y)|$ tends to $0$ as $n$ tends to infinity. By definition of the equivalent relationship in Definition \ref{Roe algebra at infinity}, the equivalent class defined by $(T^{(n)})_{n\in\IN}$ is equal to $0\in C^*_{u,\infty,\IC}((X_n)_{n\in\IN})$. This means that $a=0$, i.e., $\Psi$ is injective. This completes the proof.
\end{proof}

Let $A_{\infty}=\frac{\ell^{\infty}(X,\K)}{C_0(X,\K)}$. It is clear that $A_{\infty}$ is a $G_{\infty}(X)$-$C^*$-algebra. Combining \cite[Lemma 4.4]{STY2002}, one can show the following corollary.

\begin{Cor}
The groupoid crossed product $A_{\infty}\rtimes_r G_{\infty}(X)$ is isomorphic to $C^*_{u,\infty}((X_n)_{n\in\IN})$.\qed
\end{Cor}

Theorem \ref{boundary groupoid algebra and Roe algebra at infinity} and the corollary above provide an approach to study the Roe algebras at infinity using groupoid. If $G_{\infty}(X)$ is $K$-amenable, i.e, the canonical quotient map from the maximal groupoid $C^*$-algebra to the reduced groupoid $C^*$-algebra is a $KK$-equivalence, then one obtains that the homomorphism 
$$\lambda_*:K_*(A_{\infty}\rtimes_{\rm max}G_{\infty}(X))\to K_*(A_{\infty}\rtimes_r G_{\infty}(X))$$
induced by the canonical quotient map on $K$-theory an isomorphism. Combining with Theorem \ref{boundary groupoid algebra and Roe algebra at infinity}, the $K$-amenability of $G_{\infty}(X)$ implies Theorem \ref{K-amenable at infinity for FCE spaces}.

\begin{Rem}\label{why a-T=menable does not work}
In \cite{FS2014}, Finn-Sell showed that if $X$ admits a fibred coarse embedding into a Hilbert space, then the boundary groupoid $G_{\infty}(X)$ is a-T-menable. By Tu's result in \cite{Tu1999}, an a-T-menable groupoid is $K$-amenable. One might initially think that these results could give a more straightforward proof for Theorem \ref{K-amenable at infinity for FCE spaces}. However, this is not the case. The reason lies in the limitation of Tu's result, which can only be applied to groupoids with unit spaces that are second countable. Unfortunately, this condition does not hold for $G_{\infty}(X)$, as discussed in \cite[Remark 1.5]{JR2013}.
\end{Rem}

\section{The twisted algebras at infinity}\label{sec: twisted CBC at infty}

Let $(1\to N_n\to G_n\to Q_n\to 1)_{n\in\IN}$ be a sequence of extensions of finite groups with uniformly finite generating subsets which admits a ``FCE-by-FCE" structure. In the rest of this paper, we shall prove that the homomorphism 
$$e_*:K_*(C^*_{u,L,\infty}((P_d(G_n))_{n\in\IN}))\to K_*(C^*_{u,\infty}((P_d(G_n))_{n\in\IN}))$$
induced by evaluation map on $K$-theory is injective. 

In this section, we shall define the twisted Roe algebras and the twisted localization algebras at infinity for $P_d(G)$ using the fibred coarse embeddings of the quotient groups, and show that the evaluation map between them induces an isomorphism at the level of $K$-theory. The constructions and the idea of proofs introduced in this section are similar to those in Section \ref{subsec: Twisted algebras at infinity}. 

Let $(1\to N_n\to G_n\to Q_n\to 1)_{n\in\IN}$ be a sequence of extensions of finite groups with uniformly finite generating subsets. Let $G=\bigsqcup_{n\in\IN}G_n$ be the coarse disjoint union of $\{G_n\}_{n\in\IN}$, similarly for $N$ and $Q$.  Assume that $N$ and $Q$ admit fibred coarse embeddings into the Hilbert space $H$.

For each $d\geq 0$ and $n\in\IN$, let $P_d(Q_n)$ be the Rips complex of $Q_n$ at scale $d$ endowed with the spherical metric. For each $q\in Q_n$, denote by $Star(q)$ the open star of $q$ in the barycentric subdivision of $P_d(Q_n)$. Take a countable dense subset $Z^Q_{d,n}\subset P_d(Q_n)$ for each $d\geq 0$ in such a way that
$$Z^Q_{d,n}\subset\bigsqcup_{q\in Q_n}Star(q)\quad\text{and}\quad Z^Q_{d,n}\subset Z^Q_{d',n},\mbox{ when }d<d'.$$

For any $x\in Z^Q_{d,n}$, there exists a unique $q_x\in Q_n$ such that $x\in Star(q_x)$. In the following of this paper, we shall always assume that $q_x\in Q_n$ is the unique point such that $x\in Star(q_x)$. We denote
$$H_x=H_{q_x},\quad s(x)=s(q_x)$$
for all $x\in Z^Q_{d,n}$ and let
$$t_x(z)=t_{q_x}(q_z)$$
for all $x,z\in Z_{d,n}$, $n\in\IN$ with $q_z\in B_{Q_n}(q_x,l_n)$. We can view $P_d(Q)$ as a sparse space for each $d>0$. As $P_d(Q)$ is coarse equivalent to $Q$, there exists a sequence of non-negative numbers $0\leq \tilde l^d_0\leq \tilde l^d_1\leq\cdots\leq \tilde l^d_n\leq\cdots$ with $\lim_{n\to\infty}\tilde l^d_n = \infty$ associated with $\{l_n\}_{n\in\IN}$ for $Q_n$ and the coarse equivalence such that the field of Hilbert spaces and the trivialization above give a fibred coarse embedding of $(Z_{d,n}^Q)_{n\in\IN}$ into a Hilbert space $H$.

For each $n\in\IN$, define $V_n$ to be the finite dimensional affine subspace of $H$ spanned by $t_q (q')(s(q'))$ for all $q'\in B(q,l_n)$, $q\in Q_n$, i.e., 
$$V_n=\mbox{affine-span}\{t_q(q')(s(q'))\mid q'\in B(q,l_n),q\in Q_n\}.$$
For each $x\in Z^Q_{d,n}$, $k\geq 0$, define
$$W_k(x)=\mbox{affine-span}\{t_x(z)(s(z))\mid z\in Z_{d,n}\cap B_{P_d(Q_n)}(x,k)\}\subseteq V_n.$$
Note that for each $k\geq 0$, there exists $N\in\IN$ such that $W_k(x)$ is well defined for $x\in Z_{d,n}$ with $n\geq N$, and is an affine subspace of $V_n$. Similar with Remark \ref{composition}, the isometry $t_{xy}: V_n\to V_n$ induces a $*$-homomorphism
$$(t_{xy})_*:\A(V_n)\to\A(V_n).$$

Let $\pi:(G_n)\to (Q_n)$ be the quotient map. It induces $\pi:P_d(G_n)\to P_d(Q_n)$ by
$$\pi\left(\sum_{i=0}^{k}c_ig_i\right)=\sum_{i=0}^{k}c_i\pi(g_i)$$
where $c_i\geq 0$ and $\sum_{i=0}^{k}c_i=1$. For each $n\in\IN$, choose a countable dense subset $Z_{d,n}$ of $P_d(G_n)$ such that $\pi(Z_{d,n})=Z^Q_{d,n}$ and $Z_{d,n}\subset Z_{d',n}$ when $d<d'$.

\begin{Def}\label{twist Roe algebra for FCE-by-FCE}
For each $d\geq 0$, define $\ICu[\PdGA]$ to be the set of all equivalence classes $T=[(T^{(0)},\cdots,T^{(n)},\cdots)]$ of sequences $(T^{(0)},\cdots,T^{(n)},\cdots)$ described as follows:\begin{itemize}
\item[(1)]$T^{(n)}$ is a function from $Z_{d,n}\times Z_{d,n}$ to $\A(V_n)\wox\K$ for all $n\in\IN$ such that
$$\sup_{n\in\IN}\sup_{x,y\in Z_{d,n}}\left\|T^{(n)}(x,y)\right\|_{\A(V_n)\wox\K}<\infty;$$
\item[(2)]for any $S>0$, there exists $C>0$ such that for any bounded subset $B\subset P_d(G_n)$ with the diameter $diam(B)<S$, we have that
$$\#\{(x,y)\in B\times B\cap Z_{d,n}\times Z_{d,n}\mid T^{(n)}(x, y)\ne 0\}\leq C;$$
\item[(3)]there exists $L>0$ such that
$$\#\{y\in Z_{d,n}|T^{(n)}(x,y)\ne0\}<L,\qquad\#\{y\in Z_{d,n}\mid T^{(n)}(y,x)\ne 0\}<L$$
for all $x\in Z_{d,n}$, $n\in\IN$;
\item[(4)]there exists $R>0$ such that $T^{(n)}(x,y)=0$ whenever $d(x,y)>R$ for $x,y\in Z_{d,n}$, $n\in\IN$. The least such $R$ is called the propagation of the sequence $(T^{(0)},\cdots,T^{(n)},\cdots)$.
\item[(5)]there exists $r>0$ such that $\supp(T^{(n)}(x,y))\subseteq B_{\IR_+\times V_n}(t_{\pi(x)}(s(\pi(x))),r)$, where
$$B_{\IR_+\times V_n}(t_{\pi(x)}(s(\pi(x))),r):=\{(\tau,v)\in\IR_+\times V_n\mid\tau^2+\|v-t_{\pi(x)}(s(\pi(x)))\|^2<r^2\}$$
for all $x,y\in Z_{d,n}$, $n\in\IN$;
\end{itemize}
The equivalence relation $\sim$ on these sequences is defined by
$$(T^{(0)},\cdots,T^{(n)},\cdots)\sim(S^{(0)},\cdots,S^{(n)},\cdots)$$
if and only if
$$\lim_{n\to\infty}\sup_{x,y\in Z_{d,n}}\|T^{(n)}(x,y)-S^{(n)}(x,y)\|_{\A(V_n)\wox\K}=0.$$
\end{Def}

Comparing with Definition \ref{twisted Roe algebra for FCE}, we add condition (2) and condition (3) in Definition \ref{twist Roe algebra for FCE-by-FCE} to ensure that the operator above is bounded and locally compact. These two conditions are not necessary for Definition \ref{twisted Roe algebra for FCE} because $(X_n)_{n\in\IN}$ has bounded geometry. On the other hand, we remove Condition (2) and (5) in Definition \ref{twisted Roe algebra for FCE} for this version, because these two conditions are used to define the Dirac map. We will not need the Dirac map since we only consider the coarse Novikov conjecture. Moreover, only the fibred coarse embedding for the quotient groups $(Q_n)$ is utilized to define the ``twist" algebras.

The product structure for $\ICu[\PdGA]$ is defined as follows. For any two elements $T=[(T^{(0)},\cdots,T^{(n)},\cdots)]$ and $S=[(S^{(0)},\cdots,S^{(n)},\cdots)]$ in $\ICu[\PdGA]$, their product is defined to be
$$TS=[((TS)^{(0)},\cdots,(TS)^{(n)},\cdots)]$$
where there exists a sufficiently large $N\in\IN$ which depends on the propagation of $T$, such that $(TS)^{(n)}=0$ for $n < N$ and
$$(TS)^{(n)}(x,y)=\sum_{z\in Z_{d,n}}\left(T^{(n)}(x,z)\right)\cdot\left((t_{\pi(x)\pi(z)})_*\left(S^{(n)}(z,y)\right)\right)$$

The $*$-structure for $\ICu[\PdGA]$ is defined by
$$[(T^{(0)},\cdots,T^{(n)},\cdots)]^*=[((T^*)^{(0)},\cdots,(T^*)^{(n)},\cdots)]$$
where
$$(T^*)^{(n)}(x,y)=(t_{\pi(x)\pi(y)})_*((T^{(n)}(y,x))^*)$$
for all but finitely many $n$, and $0$ otherwise. Then $\ICu[\PdGA]$ is made into a $*$-algebra by using the additional usual matrix operations.

The norm on $\ICu[\PdGA]$ can be defined following Section \ref{subsec: Twisted algebras at infinity}. However, in the current context, it is more complicated. Let
$$\A_x\wox\K=\bigoplus_{z\in \pi^{-1}\left(B(\pi(x),\tilde l^d_n/2)\right)}\A(V_n)\wox\K.$$
For any $a_x\in \A_x\wox\K$, we denote by $a_x^z$ its projection on the $z$-th coordinate. We write $a_x=(a_x^z)_z$ for simplicity.
Notice that $\A_x\wox\K$ is an $\A(V_n)\wox\K$-module with the $\A(V_n)\wox\K$-action defined by $a_xa=(a_x^za)_{z}\in\A_x\wox\K$.

Consider $\IE_n$ to be the subset of
$$\left\{\sum_{x\in Z_{d,n}}a_x[x]\,\Big|\,a_x=(a_x^z)\in\A_x\wox\K\mbox{ and }a^z_x\ne0\mbox{ for only finite many $x,z\in Z_{d,n}$}\right\}$$
with elements $\sum_{x\in Z_{d,n}}a_x[x]$ such that there exists $n>0$ which is only determined by $\sum_{x\in Z_{d,n}}a_x[x]$ such that for any $x$, the number of $z$ with $a_x^z\ne 0$ is uniformly bounded by $n$. Equip $\IE_n$ with a pre-Hilbert-module structure over $\A(V_n)\wox\K$ as follow:
$$\left\langle\sum_{x\in Z_{d,n}}a_x[x],\sum_{x\in Z_{d,n}}b_x[x]\right\rangle=\sum_{x\in Z_{d,n}}\sum_{z\in Z_{d,n}}(a^z_x)^*b^z_x;$$
$$\left(\sum_{x\in Z_{d,n}}a_x[x]\right)a=\sum_{x\in Z_{d,n}}a_xa[x]$$
for all $a\in\A(V_n)\wox\K$ and $\sum_{x\in Z_{d,n}}a_x[x]$, $\sum_{x\in Z_{d,n}}b_x[x]\in\IE_n$. The inner product is a finite sum since there are only finite many $x$ with $a_x\ne0$ and $\#\{z\in Z_{d,n}\mid a_x^z\ne0\}$ is uniformly bounded. Denoted by $E_n$ the completion of $\IE_n$ under the norm defined by the pre-Hilbert module structure, i.e.,
$$\left\|\sum_{x\in Z_{d,n}}a_x[x]\right\|^2=\left\|\sum_{x,z\in Z_{d,n}}(a^z_x)^*b^z_x\right\|.$$
For given $T=[(T^{(0)},\cdots,T^{(n)},\cdots)]\in\ICu[\PdGA]$, let $T^{(n)}$ act on $E_n$ by
$$T^{(n)}\left(\sum_{x\in Z_{d,n}}a_x[x]\right)=\sum_{x\in Z_{d,n}}\left(\sum_{y\in Z_{d,n}}((t_{\pi(z)\pi(x)})_*T^{(n)}(x,y))a^z_y\right)[x],$$
where $((t_{\pi(z)\pi(x)})_*T^{(n)}(x,y)a^z_y)\in\A_x$ and the sum is finite for the condition (3) in Definition \ref{twist Roe algebra for FCE-by-FCE}. 

Based on the condition (1) and condition (3) in Definition \ref{twist Roe algebra for FCE-by-FCE}, one can check that $T^{(n)}$ is an adjointable module homomorphism. One can also check that the reduced norm is well-defined following the computation in Remark \ref{reduced norm is well-defined}. Furthermore, in accordance with Definition \ref{norms on twisted Roe algebra at infinity for FCE}, it is possible to define the ghost ideal $I_G((G_n,\A_n)_{n\in\IN})$. For brevity, we will omit the specific definition of it here.

\begin{Def}\label{Ax}
The twisted Roe algebras at infinity $\Cau(\PdGA)$ is defined to be the completion of $\ICu[\PdGA]$ with respect to the norm
$$\|T\|=\inf\left\{\sup_{n\in\IN}\|T^{(n)}+S^{(n)}\|_{E_n}: (S^{(n)})_{n\in\IN}\in I_G((G_n,\A_n)_{n\in\IN}))\right\},$$
where the norm of $T^{(n)}+S^{(n)}$ is given by the $*$-representation on $E_n$ as above.
\end{Def}

Now, we can also define the twisted localization algebras.

\begin{Def}\label{twist localization algebra for FCE-by-FCE}
Let $\ICuL[\PdGA]$ be the set of all bounded, uniformly norm-continuous functions
$$g:\IR_+\to \ICu[\PdGA]$$
such that $g(t)$ is of the form $g(t)=[g^{(0)}(t),\cdots,g^{(n)}(t),\cdots]$ and satisfies the following conditions\begin{itemize}
\item[(1)]there exists a bounded function $R(t):\IR_+\to \IR_+$ with $\lim\limits_{t\to\infty}R(t)=0$ such that
$$(g^{(n)}(t))(x,y)=0 ~\mbox{whenever}~ d(x,y)>R(t)~ \mbox{and}~ n\in\IN;$$
\item[(2)]there exists $r>0$ such that $\supp((g(n)(t))(x,y))\subset B_{\IR_+\times V_n}(t_{\pi(x)}(s(\pi(x))),r)$ for all $t\in\IR_+$, $x,y\in Z_{d,n}$ and $n\in\IN$.
\end{itemize}\end{Def}

\begin{Def}
The twisted localization algebra $\CauL(\PdGA)$ is defined to be the norm completion of $\ICuL[\PdGA]$ under the norm
$$\|g\|_{\infty}=\sup_{t\in\IR_+}\|g(t)\|.$$
\end{Def}

The evaluation homomorphism
$$e:\CauL(\PdGA)\to\Cau(\PdGA)$$
defined by $e(g)=g(0)$ induces a homomorphism at $K$-theory level
$$e_*:\lim_{d\to\infty}K_*(\CauL(\PdGA))\to\lim_{d\to\infty}K_*(\Cau(\PdGA))$$

Our purpose of this section is to prove the following result:

\begin{Thm}\label{twisted CBC for G_n}
Let $(1\to N_n\to G_n\to Q_n\to 1)_{n\in\IN}$ be a sequence of extensions of finite groups
with uniformly finite generating subsets. If the sequence $(N_n)_{n\in\IN}$ with the induced metric from the word metrics of $(G)_{n\in\IN}$ and the sequence $(Q_n)_{n\in\IN}$ with the quotient metrics admit fibred coarse embeddings into Hilbert spaces, then the homomorphism
$$e_*:\lim_{d\to\infty}K_*(\CauL(\PdGA))\to\lim_{d\to\infty}K_*(\Cau(\PdGA))$$
induced by the evaluation map on $K$-theory is an isomorphism.
\end{Thm}

The strategy to prove Theorem \ref{twisted CBC for G_n} combines the idea of \cite{DWY2023,GLWZ2023} with the cutting and pasting argument in Section \ref{subsec: Twisted algebras at infinity}. 

\begin{Def}
A collection $O=(O_{n,q})_{q\in Q_n, n\in\IN}$ of open subsets of $\IR_+\times V_n$, $n\in\IN$, is said to be a coherent system if for all but finitely many $n\in\IN$, and any non-empty subset $C\subseteq B_{Q_n}(q,l_n)\cap B_{Q_n}(q',l_n)$ with $q,q'\in Q_n$, we have
$$O_{n,q}\cap B_{\IR_+\times V_n}\left(t_{q}(s(q)),r\right)=t_{qq'}\left(O_{n,q'}\cap B_{\IR_+\times V_n}\left(t_{q'}(q)(s(q)),r\right)\right).$$
\end{Def}

\begin{Def}
Let $O=(O_{n,q})_{q\in Q_n,n\in\IN}$ be a coherent system of open subsets of $\IR_+\times V_n$, $n\in\IN$. For each $d\geq0$, define $\ICu[\PdGA]_{O}$ to be the $*$-subalgebra of $\ICu[\PdGA]$ generated by the equivalence classes of those sequences $[(T^{(0)},\cdots,T^{(n)},\cdots)]$ such that
$$\supp(T^{(n)}(x,y))\subseteq O_{n,q_x}$$
for all $x,y\in Z_{d,n}$ with $\pi(x)\in Star(q_x)$ for all $n\in\IN$ for some $n\in\IN$ large enough depending on the sequence $[(T^{(0)},\cdots,T^{(n)},\cdots)]$.

Define $\Cau(\PdGA)_O$ to be closure of $\ICu[\PdGA]_O$ under the norm in $\Cau(\PdGA)$ .
\end{Def}

\begin{Def}
Let $O=(O_{n,q})_{q\in Q_n,n\in\IN}$ be a coherent system of open subsets of $\IR_+\times V_n$, $n\in\IN$. For each $d\geq0$, define $\ICuL[\PdGA]_{O}$ to be the $*$-subalgebra of $\ICuL[\PdGA]$ consisting of all functions
$$g:[0,\infty)\to\ICu[\PdGA]_O$$

Define $\CauL(\PdGA)_O$ to be closure of $\ICuL[\PdGA]_O$ under the norm in $\CauL(\PdGA)$ .
\end{Def}

Then we have an evaluation homomorphism
$$e:\CauL(\PdGA)_O\to\Cau(\PdGA)_O$$
defined by $e(g)=g(0)$.

By a construction of partition of unity as in the proof of \cite[Lemma 6.3]{Yu2000} (see also \cite[Lemma 6.12]{GLWZ2022}), we have the following lemma

\begin{Lem}\label{ideal}
For any two different coherent systems $O=(O_{n,q})_{q\in Q_n,n\in\IN}$ and $O'=(O'_{n,q})_{q\in Q_n,n\in\IN}$, we have that
$$\Cau(\PdGA)_O+\Cau(\PdGA)_{O'}=\Cau(\PdGA)_{O\cup O'};$$
$$\Cau(\PdGA)_O\cap\Cau(\PdGA)_{O'}=\Cau(\PdGA)_{O\cap O'};$$
$$\CauL(\PdGA)_O+\CauL(\PdGA)_{O'}=\CauL(\PdGA)_{O\cup O'};$$
$$\CauL(\PdGA)_O\cap\CauL(\PdGA)_{O'}=\CauL(\PdGA)_{O\cap O'}.\qed$$
\end{Lem}

Fix some $0r>0$. For each $n\in\IN$, we fix a subset $\Gamma_n\subseteq Q_n$  and denote $\Gamma=(\Gamma_n)_{n\in\Gamma}$. Similar to Definition \ref{separate}, one can still define a $(\Ga,r)$-separate coherent system. We will not repeat it here.

\begin{Thm}\label{cutting and pasting}
If a coherent system $O=(O_{n,q})_{q\in Q_n,n\in\IN}$ of open subsets of $\IR_+\times V_n$, $n\in\IN$, is $(\Gamma,r)$-separate, then the evaluation homomorphism on $K$-theory
$$e_*:\lim_{d\to\infty}K_*(\CauL(\PdGA)_O)\to\lim_{d\to\infty}K_*(\Cau(\PdGA)_O)$$
is an isomorphism.
\end{Thm}

We will prove Theorem \ref{cutting and pasting} in the next subsection. Granting Theorem \ref{cutting and pasting} for the moment, we are able to prove Theorem \ref{twisted CBC for G_n}.

\begin{proof}[Proof of Theorem \ref{twisted CBC for G_n}]
For any $r>0$, we define
$$O_{n,q}^{(r)}=\bigcup_{\gamma\in B_{Q_n}(q,l_n)}B_{\IR_+\times V_n}(t_{q}(\gamma)(s(\gamma)),r)$$
for all $q\in Q_n$, $n\in\IN$. Then
$$O^{(r)}:=\left(O^{(r)}_{n,q}\right)_{q\in Q_n,n\in\IN}$$
is a coherent system of open subsets. By the definition of the twisted algebra, we have that
$$\Cau(\PdGA)=\lim_{r\to\infty}\Cau(\PdGA)_{O^{(r)}};$$
$$\CauL(\PdGA)=\lim_{r\to\infty}\CauL(\PdGA)_{O^{(r)}};$$
Consequently, it suffices to show that, for each $r>0$, the evaluation map
$$e_*:\lim_{d\to\infty}K_*(\CauL(\PdGA)_{O^{(r)}})\to\lim_{d\to\infty}K_*(\Cau(\PdGA)_{O^{(r)}})$$
is an isomorphism. Similar to the proof of Proposition\ref{twisted K-amenable at infinity for FCE spaces}, $O^{(r)}$ can be covered by finitely many $(\Ga,r)$-separate coherent systems. By using a Mayer-Vietoris argument, this theorem follows from Theorem \ref{cutting and pasting}.
\end{proof}

Before we prove Theorem \ref{cutting and pasting}, we still need some preparation similar to Definition \ref{algebra A-infty}. Let $O$ be a $(\Ga,r)$-separate coherent system for some $\Ga$, $(Y^{(n)}_{d,\gamma})_{\gamma\in\Gamma_n}$ a sequence of subspaces of $P_d(G_n)$ for any $d\geq0$, $n\in\IN$, such that
\begin{itemize}
\item the sequence $(Y_{d,\gamma})_{\gamma\in\Gamma_n,n\in\IN}$ admit fibred coarse embeddings into Hilbert spaces with uniformly controlled functions $\rho_{\pm}$;
\item $\gamma\in \pi_n(Y_{d,\gamma})$ for all $\gamma\in\Gamma_n$, $n\in\IN$, where $\pi_n:P_d(G_n)\to P_d(Q_n)$.
\end{itemize}
\begin{Def}
For any open subset $O\subseteq\IR_+\times V_n$, we denote by $\A(O)$ the $C^*$-subalgebra of $\A(V_n)$ generated by the functions whose supports are contained in $O$. Define $\Ai[\Yg]$ to be the $*$-algebra of
all equivalent class of elements $[(T^{(0)},\cdots,T^{(n)},\cdots)]$ in
$$\prod_{n\in\IN}\left(\bigoplus_{\gamma\in\Gamma_n}\IC[Y_{d,\gamma}]\wox\A(O_{n,\gamma,\gamma})\right)=\prod_{n\in\IN,\gamma\in\Gamma_n}\IC[Y_{d,\gamma}]\wox\A(O_{n,\gamma,\gamma})$$
where
$$T^{(n)}=\bigoplus_{\gamma\in\Gamma}T_{\gamma}^{(n)}$$
with
$$T^{(n)}_{\gamma}\in \IC[Y_{d,\gamma}]\wox\A(O_{n,\gamma,\gamma})$$
and the family $(T^{(n)}_{\gamma})$ satisfy the condition in Definition \ref{twist Roe algebra for FCE-by-FCE} with uniform constants.

Define $\Ai^*(\Yg)$ to be the completion of $\Ai[\Yg]$ with the norm induced from
$$\|T\|=\inf\left\{\sup_{\gamma\in\Ga}\|T^{(n)}+S^{(n)}\|: (S^{(n)})_{n\in\IN} \text{ is a ghost element}\right\}$$
for any $T=[(T^{(0)},\cdots,T^{(n)},\cdots)]\in \Ai[\Yg]$, where a ghost element in $\prod^u_{n\in\IN}C^*(Y_{d,\gamma})\wox\A(O_{n,\gamma,\gamma})$ is defined similar as before.
\end{Def}

\begin{Def}
Define $\AiL[\Yg]$ to be the $*$-subalgebra of all bounded and uniformly norm-continuous functions
$$g:[0,\infty)\to\Ai[\Yg]$$
where $g(t)$ is of the form $g(t)=[g^{(0)}(t),\cdots,g^{(n)}(t),\cdots]$ for all $t\in[0,\infty)$ and
$$f^{n}(t)=\bigoplus_{\gamma\in\Gamma_n}f^{(n)}_{\gamma}(t)$$
such that the family of $(f^{(n)}_{\gamma}(t))_{\gamma\in\Gamma_n,n\in\IN,t\geq0}$ satisfy the condition in Definition \ref{twist localization algebra for FCE-by-FCE} with uniform constants, and there exists a bounded function $R:[0,\infty)\to[0,\infty)$ with $\lim_{t\to\infty}R(t)=0$ such that
$$\left(f^{(n)}_{\gamma}(t)\right)(x,y)=0\quad \mbox{ whenever }d(x,y)>R(t)$$
for all $x,y\in Z_d,n\cap Y_{\gamma}$, $\gamma\in\Gamma_n$, $n\in\IN$, $t\in[0,\infty)$.

Define $\AiL^*(\Yg)$ to be closure of $\AiL[\Yg]$ under the norm
$$\|g\|=\sup_{t\in[0,\infty)}\|g(t)\|.$$
\end{Def}

Denoted by
$$\mathcal N_{P_d(G_n)}(Y_{d,\gamma},S)=\{x\in P_d(G_n)\mid d_{P_d(G_n)}(x,Y_{d,\gamma})\leq S\}$$
the $S$-neighbourhood of $Y_{d,\gamma}$ in $P_d(G_n)$ for $n\in\IN$.

\begin{Pro}\label{cut}
Suppose that $O=(O_{n,q})_{q\in Q_n,n\in\IN}$ is a coherent system of open subsets of $\IR_+\times V_n$, $n\in\IN$ which is $(\Gamma,r)$-separate for some $\Gamma=\left(\Gamma_n\right)_{n\in\IN}$ and $r>0$ as above. Then\begin{itemize}
\item[(1)]$\Cau(\PdGA)_O\cong\lim\limits_{S\to\infty}\Ai^*((\mathcal N_{P_d(G_n)}(\pi^{-1}(\gamma),S):\gamma\in\Gamma_n)_{n\in\IN})$;
\item[(2)]$\CauL(\PdGA)_O\cong\lim\limits_{S\to\infty}\AiL^*((\mathcal N_{P_d(G_n)}(\pi^{-1}(\gamma),S):\gamma\in\Gamma_n)_{n\in\IN})$.
\end{itemize}\end{Pro}

\begin{proof}
We shall establish an isomorphism only for the first item. Take arbitrarily an element
$$T=[(T^{(0)},\cdots,T^{(n)},\cdots)]\in\ICu[\PdGA]_O$$
where $T^{(n)}:Z_{d,n}\times Z_{d,n}\to\A(V_n)\wox\K$ is a function such that there exists $S_0>0$ satisfying
$$\supp(T^{(n)}(x,y))\subseteq O_{n,x}\subseteq\bigsqcup_{\gamma\in\Gamma_n\cap B_{Q_n}(\pi(x),l_n)}O_{n,x,\gamma}$$
for all $x,y\in Z_{d,n}$, $n\in\IN$. Since the coherent system $O$ is $(\Gamma,r)$-separate. Then we have a direct sum decomposition
$$T^{n}(x,y)=\bigoplus_{\gamma\in\Gamma_n\cap B_{Q_n}(\pi(x),l_n)}T^{(n)}_{\gamma}(x,y)$$
where
$$T^{(n)}_{\gamma}(x,y)=T^{(n)}(x,y)|_{O_{n,x,\gamma}}\in\A(O_{n,x,\gamma})\wox\K.$$
On the other hand, it follows from the support condition in Definition \ref{twist Roe algebra for FCE-by-FCE}, there exists $S_0>0$ such that
$$\supp(T^{(n)}(x,y))\subseteq B_{\IR_+\times V_n}(t_{\pi(x)}(\pi(x))(s(\pi(x))),S_0)$$
Hence, $T^{(n)}(x,y)=0$ whenever
$$d(t_{\pi(x)}(\pi(x))(s(\pi(x))),t_{\pi(x)}(\gamma)(s(\gamma)))>S_0+r$$
It follows that there exists $S>0$ such that $T^{(n)}_{\gamma}(x,y)=0$ whenever $d_{P_d(Q_n)}(\pi(x),\gamma)>S$ for all $x,y\in Z_{d,n}$, $\gamma\in\Gamma_n$, $n\in\IN$. Define
$$U^{(n)}_{\gamma}(x,y)=(t_{\gamma\pi(x)})_*(T^{(n)}_{\gamma}(x,y)).$$
Since $T$ has finite propagation, $U_{\gamma}^{(n)}(x,y)$ is well-defined for all large enough $n\in\IN$, and
$$U_{\gamma}^{(n)}\in\A(O_{n,\gamma,\gamma})\wox\K$$
for all $x,y\in Z_{d,n}\cap B_{Q_n}(\gamma,S)$ with $\gamma\in\Gamma_n\cap B_{Q_n}(\pi(x),l_n)$ and $n\in\IN$. Therefore, we have
$$U_{\gamma}^{(n)}\in\ICu[\mathcal N_{P_d(G_n)}(\pi^{-1}(\gamma),S)]\wox\A(O_{n,\gamma,\gamma}).$$
Define
$$U^{(n)}=\bigoplus_{\gamma\in\Gamma_n}U^{(n)}_{\gamma}\in\bigoplus_{\gamma\in\Gamma_n}\IC[\mathcal N_{P_d(G_n)}(\pi^{-1}(\gamma),S)]\wox\A(O_{n,\gamma,\gamma})$$
for all $n\in\IN$ but finitely many $n$ and $0$ otherwise. Then the class $U=[(U^{(0)},\cdots,U^{(n)},\cdots)]$ is an element of $\Ai[(\mathcal N_{P_d(G_n)}(\pi^{-1}(\gamma),S):\gamma\in\Gamma_n)_{n\in\IN}]$. Then the correspondence $T\to U$ extends to a $*$-isomorphism.
\end{proof}

As each $N_n$ is a normal subgroup of $G_n$ for each $n\in\IN$, for a fixed $S>0$, the sequence $\mathcal N_{P_d(G_n)}(\pi^{-1}(\gamma),S)$ is uniformly coarse equivalent to the sequence $N=((N_n)_{n\in\IN})$. Thus we can conclude the following result.

\begin{Thm}\label{FCE CBC infty}
Let $S> 0$ be any positive number. Then the map
$$e_*:\lim_{d\to\infty}K_*(\AiL^*((\mathcal N_{P_d(G_n)}(\pi^{-1}(\gamma),S):\gamma\in\Gamma_n)_{n\in\IN}))\to\lim_{d\to\infty}K_*(\Ai^*((\mathcal N_{P_d(G_n)}(\pi^{-1}(\gamma),S):\gamma\in\Gamma_n)_{n\in\IN}))$$
induced by the evaluation-at-zero map on $K$-theory is an isomorphism.
\end{Thm}

\begin{proof}
From the definition, one has that the coarse Baum-Connes conjecture at infinity for the coarse disjoint union of a sequence of finite metric spaces is equivalent to the coarse Baum-Connes conjecture at infinity for the separate disjoint union of a sequence of finite metric spaces. Recall that a metric space $X=\bigsqcup_{n\in\IN}X_n$ is called a separate disjoint union if $d$ is equal to $d_n$ when restricting on $X_n$, and $d(X_n,X_m)=\infty$ if $n\ne m$. Notice that the \emph{algebraic Roe algebra at infinity} for separate disjoint union $X=\bigsqcup_{n\in\IN}X_n$ is defined as a quotient algebra of $\prod^u_{n\in\IN}C^*(X_n)$ associated with the ghost ideal.

By assumption, $(N_n)_{n\in\IN}$ admits a fibred coarse embedding into a Hilbert space. For each $\gamma\in\Gamma_n$, $n\in\IN$, set $N_{\gamma}=N_n$ and $l_{\gamma}=$. Then the sequence $(N_{\gamma})_{\gamma\in\Ga}$ also admits a fibred coarse embedding into a Hilbert space. Fixed $S>0$. The sequence of space $\{\mathcal N_{P_d(G_n)}(\pi^{-1}(\gamma),S)\}_{\gamma\in\Ga_n,n\in\IN}$ is uniformly coarsely equivalent to the sequence $\{P_d(\pi^{-1}(\gamma))\}_{\gamma\in\Ga_n,n\in\IN}$. Since $N_n$ is a normal subgroup, $\pi^{-1}(\gamma)$ is isometric to $N_n$ for each $\gamma\in\Ga_n\subseteq Q_n$. Since the coarse evaluation map is coarsely invariant (see \cite[Section 7.1]{HIT2020}), it suffices to show $e_*$ is an isomorphism for $(N_{\gamma})_{\gamma\in\Ga}$, which follows directly from Corollary \ref{CBC at infinity for FCE} and the fact that $(N_{\gamma})_{\gamma\in\Ga}$ admits a fibred coarse embedding into a Hilbert space. This finishes the proof.
\end{proof}

Now, we are ready to prove Theorem \ref{cutting and pasting}.

\begin{proof}[Proof of Theorem \ref{cutting and pasting}]
For brevity, we set up
\begin{equation*}\begin{split}
AL(d,S)&=\AiL^*((\mathcal N_{P_d(G_n)}(\pi^{-1}(\gamma),S):\gamma\in\Gamma_n)_{n\in\IN})),\\
A(d,S)&=\Ai^*((\mathcal N_{P_d(G_n)}(\pi^{-1}(\gamma),S):\gamma\in\Gamma_n)_{n\in\IN})),\\
CL_d&=\CauL(\PdGA)_O,\\
C_d&=\Cau(\PdGA)_O.
\end{split}\end{equation*}
Then by Proposition \ref{cut}, we have the following commutative diagram
$$\xymatrix{
\lim\limits_{d\to\infty}K_*(CL_d)\ar[r]^{e_*}\ar[d]_{\cong}&\lim\limits_{d\to\infty}K_*(C_d)\ar[d]^{\cong}\\
\lim\limits_{d\to\infty}\lim\limits_{S\to\infty}K_*(AL(d,S))\ar[r]^{e_*}\ar[d]_{\cong}&\lim\limits_{d\to\infty}\lim\limits_{S\to\infty}K_*(A(d,S))\ar[d]^{\cong}\\
\lim\limits_{S\to\infty}\lim\limits_{d\to\infty}K_*(AL(d,S))\ar[r]^{\cong}&\lim\limits_{S\to\infty}\lim\limits_{d\to\infty}K_*(A(d,S)).}$$
The second line and the third line commute for the following reasons. Notice that we have the following commutative diagram
$$\xymatrix{
K_*(AL(d,S))\ar[rd]\ar[rrr]^{e_*}\ar[dd]&&&K_*(A(d,S))\ar[rd]\ar[dd]&\\
&K_*(AL(d,S'))\ar[rrr]^{e_*}\ar[dd]&&&K_*(A(d,S'))\ar[dd]\\
K_*(AL(d',S))\ar[rd]\ar[rrr]^{e_*}&&&K_*(A(d',S))\ar[rd]&\\
&K_*(AL(d',S'))\ar[rrr]^{e_*}&&&K_*(A(d',S'))\\}$$
for all $d'>d$, $S'>S$, where $AL(d,S)\to AL(d',S)$, $AL(d,S)\to AL(d,S')$ are given by the inclusions. It allows us to change the order of limits from $\lim_{d\to\infty}\lim_{S\to\infty}$ to $\lim_{S\to\infty}\lim_{d\to\infty}$. So by Theorem \ref{FCE CBC infty}, the third horizontal map is an isomorphism. Then it follows that
$$e_*:\lim_{d\to\infty}K_*(\CauL(\PdGA)_O)\to\lim_{d\to\infty}K_*(\Cau(\PdGA)_O)$$
is an isomorphism, since all vertical maps are isomorphisms.
\end{proof}

\section{The Bott maps and proof of the main result}\label{sec: Proof of main result}

In this section, we shall define the Bott maps $\beta$, $\beta_L$ to construct the following commutative diagram
$$\xymatrix{
K_*(\CauL(\PdG))\ar[d]_{(\beta_L)_*}\ar[r]^{e_*}&K_*(\Cau(\PdG))\ar[d]^{(\beta)_*}\\
K_*(\CauL(\PdGA))\ar[r]^{e_*}&K_*(\Cau(\PdGA)).}$$
We will show that the Bott map $(\beta_L)_*$ is an isomorphism on $K$-theory. The maps $\beta$ and $\beta_L$ are constructed by asymptotic morphisms which play the roles of the geometric analoglue of Bott periodicity theorem proved by Yu in \cite{Yu2000}.

Now, we define the Bott maps $\beta$ and $\beta_L$, and they are defined following Definition \ref{Bott for FCE}. For each $n\in\IN$ and $x\in Z_{d,n}$, the inclusion of the $0$-dimensional affine subspace $\{t_{\pi(x)}(\pi(x))(s(\pi(x)))\}$ into $V_n$ induces a $*$-homomorphism
$$\beta(x):\S\cong\A(\{t_{\pi(x)}(\pi(x))(s(\pi(x)))\})\to\A(V_n)$$
by the formula
$$\left(\beta(x)\right)(g)=g\left(X\wox 1+1\wox C_{V_n,t_{\pi(x)}(\pi(x))(s(\pi(x)))}\right)$$
where
$$C_{V_n,t_{\pi(x)}(\pi(x))(s(\pi(x)))}:V_n\to V_n^0\subseteq\Cl(V_n^0)$$
is the unbounded function $v\mapsto v- t_{\pi(x)}(\pi(x))(s(\pi(x)))\in V_n^0\subseteq\Cl(V_n^0)$ for all $v\in V_n$.

\begin{Def}
For each $d\geq0$ and $t\in [1,\infty)$, define a map
$$\beta_t:\S\wox\ICu[\PdG]\to\Cau(\PdGA)$$
by the formula
$$\beta_t(g\wox T)=[((\beta_t(g\wox T))^{(0)},\cdots,(\beta_t(g\wox T))^{(n)},\cdots)]$$
for each $g\in\S$, $T=[(T^{(0)},\cdots,T^{(n)},\cdots)]\in\ICu[\PdG]$, where
$$(\beta_t(g\wox T))^{(n)}(x,y)=\left(\beta(x)\right)(g_t)\wox T^{(n)}(x,y)$$
for $x,y\in Z_{d,n}$, $n\in\IN$, and $g_t(r)=g(\frac rt)$ for all $r\in\IR$.
\end{Def}

\begin{Def}
For each $d\geq0$ and $t\in [1,\infty)$, define a map
$$(\beta_L)_t:\S\wox\ICuL[\PdG]\to\CauL(\PdGA)$$
by the formula
$$((\beta_L)_t(f))(r)=\beta_t(f(r))$$
for $f\in\S\wox\ICuL[\PdG]$ and $r\in\IR_+$.
\end{Def}

We have the following result which is proved in a similar manner to Lemma \ref{Bott for FCE}.

\begin{Lem}\label{Bott}
For each $d\geq0$, the maps $\beta_t$ and $(\beta_L)_t$ extend respectively to asymptotic morphisms
$$\beta:\S\wox\Cau(\PdG)\leadsto\Cau(\PdGA)$$
$$\beta_L:\S\wox\CauL(\PdG)\leadsto\CauL(\PdGA).\qed$$
\end{Lem}

Note that the asymptotic morphisms
$$\beta:\S\wox\Cau(\PdG)\leadsto\Cau(\PdGA)$$
$$\beta_L:\S\wox\CauL(\PdG)\leadsto\CauL(\PdGA)$$
induce homomorphisms on $K$-theory
$$\beta_*:K_*(\Cau(\PdG))\to K_*(\Cau(\PdGA))$$
$$(\beta_L)_*:K_*(\CauL(\PdG))\to K_*(\CauL(\PdGA)).$$

\begin{Thm}\label{BottL}
For any $d> 0$, the Bott map
$$(\beta_L)_*:K_*(\CauL(\PdG))\to K_*(\CauL(\PdGA))$$
is an isomorphism.
\end{Thm}

\begin{proof}
The $K$-theory of the localization algebra is invariant under the strongly Lipschtiz homotopy equivalence (see \cite{Yu1997,Yu2000}). By the Mayer-Vietoris sequence and induction on dimensions (cf. \cite{Yu1997,SW2007,CWY2013}), the general case can be reduced to the zero-dimensional case, i.e., if $D\subset P_d(G)$ is a $\delta$-separated subspace for some $\delta>0$, then
$$(\beta_L)_*:K_*(\CauL((D_n)_{n\in\IN}))\to K_*(\CauL((D_n,\A(V_n))_{n\in\IN}))$$
is an isomorphism. Let $D_n=D\cap P_d(G_n)$.

Notice that
$$K_*(\CauL((D_n)_{n\in\IN}))\cong\frac{\prod_{n=0}^{\infty}\prod_{x\in D_n}K_*(C^*_L(\{x\}))}{\bigoplus_{n=0}^{\infty}\prod_{x\in D_n}K_*(C^*_L(\{x\}))},$$
$$K_*(\CauL((D_n,\A(V_n))_{n\in\IN}))\cong\frac{\prod_{n=0}^{\infty}\prod_{x\in D_n}K_*(C^*_L(\{x\},\A(V_n)))}{\bigoplus_{n=0}^{\infty}\prod_{x\in D_n}K_*(C^*_L(\{x\},\A(V_n)))}.$$
Then the theorem for zero-dimensional case follows from that $(\beta_L)_*$ restricts to an isomorphism from $K_*(C^*(\{x\}))\cong K_*(\K)$ to $\K_*(C^*_L(\{x\},\A(V_n)))\cong K_*(\K\otimes\A(V_n))$
at each $x\in D$ by the classical Bott periodicity.
\end{proof}

\begin{proof}[Proof of Theorem \ref{CNC at infty}]
It follows from Lemma \ref{Bott} the following diagram commutes:
$$\xymatrix{
K_*(\CauL(\PdG))\ar[d]_{(\beta_L)_*}^{\cong}\ar[r]^{e_*}&K_*(\Cau(\PdG))\ar[d]^{(\beta)_*}\\
K_*(\CauL(\PdGA))\ar[r]^{e_*}_{\cong}&K_*(\Cau(\PdGA)).}$$
By Theorem \ref{twisted CBC for G_n} and Theorem \ref{BottL}, we have that $(\beta_L)_*$ and the second horizontal map are both isomorphisms. The proof of theorem is completed through diagram chasing.
\end{proof}

\bibliographystyle{alpha}
\bibliography{ref}

\end{document}